\documentclass[reqno]{amsart}

\usepackage{latexsym,amsthm,mathrsfs,amsmath,amscd,enumerate,enumitem, amscd,color}
\usepackage[cspex,bbgreekl]{mathbbol}
\usepackage{amsfonts}
\usepackage{amssymb}

\DeclareSymbolFontAlphabet{\amsbb}{bbold}
\DeclareSymbolFontAlphabet{\mathbb}{AMSb}%

 \newtheorem{thm}{Theorem}[section]
 
 \newtheorem{lem}[thm]{Lemma}
 \newtheorem{prop}[thm]{Proposition}
\theoremstyle{definition}

 \theoremstyle{remark}
 \newtheorem{rem}[thm]{Remark}

\newcommand{\supp}{\mathop{\mathrm{supp}}}

\DeclareFontEncoding{FMS}{}{}
\DeclareFontSubstitution{FMS}{futm}{m}{n}
\DeclareFontEncoding{FMX}{}{}
\DeclareFontSubstitution{FMX}{futm}{m}{n}
\DeclareSymbolFont{fouriersymbols}{FMS}{futm}{m}{n}
\DeclareSymbolFont{fourierlargesymbols}{FMX}{futm}{m}{n}
\DeclareMathDelimiter{\VERT}{\mathord}{fouriersymbols}{152}{fourierlargesymbols}{147}

\numberwithin{equation}{section}
\allowdisplaybreaks

\begin{document}
\title[]
 {Bellman functions and dimension free $L^p$--estimates for the Riesz transforms in Bessel settings}

\author[J. J. Betancor]{Jorge J. Betancor}
\address{Jorge J. Betancor, Juan C. Fari\~na\newline
    Departamento de An\'alisis Matem\'atico, Universidad de La Laguna,\newline
    Campus de Anchieta, Avda. Astrof\'isico S\'anchez, s/n,\newline
    38721 La Laguna (Sta. Cruz de Tenerife), Spain}
\email{jbetanco@ull.es, jcfarina@ull.es}

\author[E. Dalmasso]{Estefan\'ia Dalmasso}
\address{Estefan\'ia Dalmasso\newline
    Instituto de Matem\'atica Aplicada del Litoral, UNL, CONICET, FIQ.\newline Colectora Ruta Nac. NÂ° 168, Paraje El Pozo,\newline S3007ABA, Santa Fe, Argentina}
\email{edalmasso@santafe-conicet.gov.ar}

\author[J. C. Fari\~na]{Juan C. Fari\~na}

\author[R. Scotto]{Roberto Scotto}
\address{Roberto Scotto\newline
    Universidad Nacional del Litoral, FIQ.\newline Santiago del Estero 2829,\newline S3000AOM, Santa Fe, Argentina}
\email{roberto.scotto@gmail.com}

\thanks{The first and the third authors are partially supported by MTM2016-79436-P. {The second author is partially supported by PIP 2015-900 (CONICET). The fourth author is supported by CAI+D (UNL)}}
\date{\today}
\subjclass[2010]{42A50, 42B20, 42B25.}

\keywords{Riesz transforms, Bessel operators, Bellman functions, dimension free, bilinear embeddings.}

\date{}

\begin{abstract}
In this article we prove dimension free $L^p$--boundedness of Riesz transforms associated with a Bessel differential operator. We obtain explicit estimates of the $L^p$--norms for the Bessel--Riesz transforms in terms of $p$, establishing a linear behaviour with respect to $p$. We use the Bellman function technique to prove a bilinear dimension free inequality involving Poisson semigroups defined through this Bessel operator.
\end{abstract}
\maketitle

\section{Introduction and main results}
For every $j=1,\dots, d$, the classical Riesz transform $R_jf$ of $f\in L^p(\mathbb R^d)$, $1\leq p<\infty$, is defined by
\[R_jf(x)=\frac{\Gamma(\frac{d+1}{2})}{\pi^{\frac{d+1}{2}}} \lim\limits_{\epsilon \rightarrow 0^+} \int_{|x-y|>\epsilon} \frac{x_j-y_j}{|x-y|^{d+1}}f(y)dy,\quad a.e.\, x=(x_1,\dots,x_d)\in \mathbb R^d.\]
As it is well-known, $R_j$ is bounded on $L^p(\mathbb R^d)$ for every $1< p<\infty$ and $j=1,\dots, d$ with constant independent of dimension. Stein \cite[Theorem,~p.~71]{St} proved something stronger than that: for every $1<p<\infty$, there exists $A_p>0$ such that $\|\mathbf{R}f\|_{L^p(\mathbb R^d)}\leq A_p \|f\|_{L^p(\mathbb R^d)}$ for $f\in L^p(\mathbb R^d)$, where $\mathbf{R}f=\left(\sum_{j=1}^d |R_jf|^2\right)^{1/2}$.
By using transference methods, Duoandikoetxea and Rubio de Francia \cite{DR} gave another proof of Stein's result.

Let $\mathcal H$ denote the Hilbert transform on $\mathbb R$. Gokhberg and Krupnik \cite{GK} and Pichorides \cite{Pi} proved that
\[\|\mathcal H\|_{L^p(\mathbb R^d)\rightarrow L^p(\mathbb R^d)}=\left\{ \begin{array}{lc}
    \tan(\frac{\pi}{2p}), & 1<p\leq 2,\\
    \cot(\frac{\pi}{2p}), & p\geq 2.
    \end{array}\right.\]
Iwaniev and Martin \cite{IM} extended this result to higher dimensions proving that, for every $j=1,\dots,d$, $\|R_j\|_{L^p(\mathbb R^d)\rightarrow L^p(\mathbb R^d)}=\|\mathcal H\|_{L^p(\mathbb R^d)\rightarrow L^p(\mathbb R^d)}$, $1<p<\infty$. This result was also showed by Ba\~nuelos and Wang \cite{BW} using probabilistic methods.

This phenomenon of dimension free estimates has been observed for
Riesz transforms appearing in other settings. For instance, Coulhon,
Muller and Zienkiewicz \cite{CMZ} worked with Riesz transforms on
the Heisenberg group; Lust-Piquard \cite{Lu1} considered discrete
Riesz transforms on a product of Abelian groups and Urban and
Zienkiewicz \cite{UZ} proved these estimates for Riesz transforms
associated with Schr\"o- dinger operators.

In recent years, finding dimension free $L^p$--estimates for Riesz
transforms associated with orthogonal systems has been on the rise
considerably. In the case of Hermite polynomials which are
associated with the Ornstein--Uhlenbeck differential operator and
the Gaussian measure, this kind of results were obtained by Pisier
\cite{Pis} and Guti\'errez \cite{Gut} (see also \cite{Gun} and
\cite{Me} for proofs using probabilistic methods). Harboure, de
Rosa, Segovia and Torrea \cite{HRST} proved $L^p$-dimension free
boundedness for Riesz transforms related to the harmonic oscillator
and the Hermite functions (see \cite{Lu2} for another proof). In
different settings associated with Laguerre operators, the results
have been obtained by Nowak \cite{No} and Stempak and Wr\'obel
\cite{SW}. A general result that applies to different orthogonal
systems was proved by Forzani, Sasso and Scotto \cite{FSS}. In
\cite{FSS}, \cite{Gut}, \cite{HRST}, \cite{Lu2}, \cite{No}, and
\cite{SW} a procedure based upon the boundedness of Littlewood-Paley
functions is used and it is inspired by the Stein's method of proof
developed in \cite{St}.

Recently, Dragi\v{c}evic and Volberg in a series of papers
(\cite{DV1}, \cite{DV2}, \cite{DV4}, and \cite{DV3}) have developed
another approach, different from any previous ones, where they use a
technique involving Bellman functions. These Bellman functions allow
them to obtain dimension free bilinear Littlewood--Paley estimates.
Then, as a consequence, dimension free $L^p$-estimates for Riesz
transforms are also obtained. Dragi\v{c}evic and Volberg have
studied classical and Gaussian (Hermite polynomial) context in
\cite{DV1}. Riesz transforms associated with the harmonic oscillator
and more general Schr\"odinger operators were analyzed in \cite{DV2}
and \cite{DV3}, respectively. In \cite{DV4}, bilinear
Littlewood--Paley estimates related to elliptic operators were
established.

By using the Bellman function techniques it is possible to get dimension free $L^p$-estimates with better constants than the ones given by other procedures. Furthermore, it is proved that these constants are linear in $\max\{p,p/(p-1)\}$ for $1<p<\infty$.

Bellman functions have also been  utilized in the boundedness of Riesz transforms appearing in other settings, such as: Bakry context (\cite{CD} and \cite{Da}), Hodge--Laguerre operators (\cite{MS}), discrete orthogonal systems (\cite{W1} and \cite{W2}), weighted Riesz transforms (\cite{DPW}) and Beurling operator (\cite{PV}).

In this paper we obtain dimension free $L^p$-estimates for Riesz transforms in this Bessel setting by using Bellman functions techniques.

Muckenhoupt and Stein (\cite{MuSt}) started with the study of
harmonic analysis associated to Bessel operators. They considered,
for {$\alpha\geq 0$}, the Bessel operator
\[B_\alpha=-x^{2\alpha} \frac{d}{dx} x^{-2\alpha}\frac{d}{dx}=-\frac{d^2}{dx^2}+\frac{2\alpha}{x}\frac{d}{dx},\]
and the Hankel transformation $h_\alpha$ defined, for every $f\in L^1((0,\infty),x^{2\alpha}dx)$, by
\[h_\alpha f(x)=\int_0^\infty \phi_y^\alpha(x) f(y) x^{2\alpha} dx, \quad x\in (0,\infty),\]
where $\phi_y^\alpha(x)=(xy)^{-\alpha+1/2} J_{\alpha-1/2}(xy)$ are
the eigenfunctions of $B_\alpha$, and $J_\mu$ represents the Bessel
function of first kind and order $\mu$. Notice that
$\phi_y^\alpha(x)=\phi_x^\alpha(y)$, $x,y\in (0,\infty)$, and
$\phi_y^\alpha \in L^p((0,\infty), x^{2\alpha} dx)$ if and only if
$\frac{2\alpha+1}{\alpha}<p\le\infty$, when $\alpha>0${, and
$p=\infty$, when $\alpha=0$} (see \cite[(5.16.1), p. 134]{Leb}).

Since the function $z^{-\mu}J_\mu(z)$ is bounded on $(0,\infty)$ for
$\mu\ge -1/2$ (see \cite[(5.16.1), p. 134]{Leb}), $h_\alpha$ defines a
bounded operator from $L^1((0,\infty),$ $x^{2\alpha}dx)$ to
$L^\infty((0,\infty),x^{2\alpha}dx)=L^\infty((0,\infty),dx)$. Also,
$h_\alpha$ can be extended from $L^1((0,\infty),$ $ x^{2\alpha}dx)\cap
L^2((0,\infty),x^{2\alpha}dx)$ to $L^2((0,\infty),x^{2\alpha}dx)$
as an isometry on $L^2((0,\infty),x^{2\alpha}dx)$ and
$h_\alpha^{-1}=h_\alpha$ (\cite{Hir}). Let us note that both the
Hankel transformation and the Bessel operator are connected by
\begin{equation}\label{HankBes}
h_\alpha(B_\alpha f)(x)=x^2 h_\alpha(f)(x), \quad x\in (0,\infty)
\end{equation}
for every $f\in C_c^\infty(0,\infty)$, the space of smooth functions with compact support in $(0,\infty)$.

Muckenhoupt and Stein (\cite{MuSt}) introduced the Riesz transform $R_\alpha$ associated with $B_\alpha$ as follows. For every $f\in L^2((0,\infty),x^{2\alpha}dx)$,
\[R_\alpha f(x)=-x h_{\alpha+1}\left(\frac{1}{y} h_\alpha f(y)\right)(x).\]
Since $h_\alpha$ is bounded on $L^2((0,\infty),x^{2\alpha}dx)$, it
is clear that $R_\alpha$ is also a bounded operator on
$L^2((0,\infty),x^{2\alpha}dx)$.  As in the classical case,
$R_\alpha$ can be extended from $L^2((0,\infty),x^{2\alpha}dx)\cap
L^p((0,\infty),x^{2\alpha}dx)$ to $L^p((0,\infty),x^{2\alpha}dx)$ as
an $L^p$--bounded operator, for every $1<p<\infty$, and from $L^1((0,\infty),x^{2\alpha}dx)$ into $L^{1,\infty}((0,\infty),x^{2\alpha}dx)$ (\cite{MuSt}). On the other hand,
$ L^p$-weighted inequalities for $R_\alpha$ were established in \cite{AK} and, more recently, in \cite{BHNV}.

We define $\mathcal A_{\alpha,c}(0,\infty)$ as the space consisting
of all those functions $\phi \in C^\infty (0,\infty)$ $\cap
L^1((0,\infty),x^{2\alpha}dx)$ such that $h_\alpha(\phi)\in
C_c^\infty(0,\infty)$. Since $h_\alpha$ is an isometry in
$L^2((0,\infty),$ $ x^{2\alpha}dx)$, the space $\mathcal
A_{\alpha,c}(0,\infty)$ is a dense subspace of
$L^2((0,\infty),x^{2\alpha}dx)$.

Motivated by \eqref{HankBes} we define, for every $f\in \mathcal A_{\alpha,c}(0,\infty)$, $B_\alpha^{-1/2}f$ by
\[B_\alpha^{-1/2}f=h_\alpha \left( \frac{1}{y}h_\alpha f(y) \right).\]
Then, we have that for every $f\in \mathcal A_{\alpha,c}(0,\infty)$
\begin{equation}\label{Ralfa}
R_\alpha f(x)=\frac{d}{dx} B_\alpha^{-1/2}f(x), \quad x\in (0,\infty).
\end{equation}
 Equality \eqref{Ralfa} says that $R_\alpha$ is the Riesz transform associated with $B_\alpha$ in the sense of Stein \cite{StLP}. Furthermore, $R_\alpha$ can be seen as a principal value integral operator. For every $f\in L^p((0,\infty),x^{2\alpha} dx)$, $1\leq p<\infty$, we can write
\[R_\alpha f(x)=\lim\limits_{\epsilon \rightarrow 0^+} \int_{|x-y|>\epsilon} R_\alpha(x,y) f(y) y^{2\alpha} dy, \quad a.e.\; x\in (0,\infty),\]
where
\[R_\alpha(x,y)= \int_0^\infty \partial_x \mathcal P_t^\alpha(x,y) dt, \quad x,y\in (0,\infty),\, x\neq y,\]
and $\mathcal P_t^\alpha$ denotes the Poisson kernel associated to $B_\alpha$ (see Section 2 for definitions). The Bessel--Riesz transform $R_\alpha$ turns out to be a Calder\'on--Zygmund operator in the homogeneous type space $((0,\infty), x^{2\alpha}dx, \rho)$ where $\rho$ represents the Euclidean metric on $(0,\infty)$ (see \cite{BCN}).

Bessel--Riesz transforms in higher dimensions were considered in
\cite{BCC}. Given a multi-index $\alpha=(\alpha_1,\dots,\alpha_d)\in
{\mathbb R^d_{\geq 0}:=[0,\infty)^d}$, let us define the
$d$-dimensional Bessel differential operator as
\[B_\alpha=\sum\limits_{i=1}^d B_{\alpha_i,x_i} = \sum\limits_{i=1}^d \left(-\frac{\partial^2}{\partial x_i^2}+\frac{2\alpha_i}{x_i}\frac{\partial}{\partial x_i}\right).\]
The Hankel transform $h_\alpha(f)$ of $f$ is defined, for $x\in \mathbb R^d_+$, by
\[h_\alpha f(x)=\int_{\mathbb R^d_+} \phi_y^\alpha (x) f(y) y^{2\alpha} dy,\]
where
\[\phi_y^\alpha (x)=\prod\limits_{j=1}^d (x_j y_j)^{-\alpha_j+1/2} J_{\alpha_j-1/2}(x_j y_j),\quad x=(x_1,\dots,x_d), \, y=(y_1,\dots,y_d)\in \mathbb{R}_+^d,\]
 $y^{2\alpha}=\prod_{j=1}^d y_j^{2\alpha_j}$
and $f\in L^1(\mathbb R^d_+,x^{2\alpha} dx)$. For every $y\in
\mathbb{R}_+^d$, the function $\phi_y^\alpha$ is an eigenfunction of
$B_\alpha$, being $B_\alpha\phi_y^\alpha=|y|^2\phi_y^\alpha$. The
Hankel transform $h_\alpha$ can be extended from $L^1(\mathbb
R^d_+,x^{2\alpha} dx)\cap L^2(\mathbb R^d_+,x^{2\alpha} dx)$ to
$L^2(\mathbb R^d_+,x^{2\alpha} dx)$ as an isometry on $L^2(\mathbb
R^d_+,x^{2\alpha} dx)$ and $h_\alpha^{-1}=h_\alpha$, as in the
one-dimensional case.

Let $i=1,\dots,d$ be given. For every $f\in L^2(\mathbb R^d_+,x^{2\alpha} dx)$, the $i$-th Bessel-Riesz transform $R_{\alpha,i}f$ is defined by
\[R_{\alpha,i} f(x)=-x_i h_{\alpha+e_i}\left(\frac{1}{|y|}h_\alpha f(y)\right)(x), \quad x=(x_1,\dots,x_d)\in \mathbb R^d_+,\]
where $e_i$ is the $i$-th unit vector on $\mathbb R^d_+$. The
operator $R_{\alpha,i}$ is bounded on $L^2(\mathbb
R^d_+,x^{2\alpha}$  $dx)$.

By $S(\mathbb{R}_+^d)$ we represent the space consisting of all
those functions $\phi\in C^\infty(\mathbb{R}_+^d)$ such that, for
every $m\in \mathbb{N}_0=\mathbb{N}\cup \{0\}$ and
$k=(k_1,\ldots,k_d)\in \mathbb{N}_0^d$,
\[
\gamma_{m,k}(\phi)=:\sup_{x\in
\mathbb{R}_+^d}|x|^m\left|\frac{\partial^{k_1+\ldots+k_d}}{\partial
x_1\ldots\partial x_d}\phi(x)\right|<\infty.
\]
$S(\mathbb{R}_+^d)$ is endowed with the topology generated by the
system $\{\gamma_{m,k}\}_{m\in \mathbb{N}_0,k\in \mathbb{N}_0^d}$ of
seminorms. The Hankel transformation $h_\alpha$ is an automorphism
in $S(\mathbb{R}_+^d)$ (\cite{Al}). For every $\phi\in
S(\mathbb{R}_+^d)$, we have that
\begin{equation}\label{oper}
h_\alpha(B_\alpha\phi)(x)=|x|^2h_\alpha(\phi)(x),\,\,\,x\in
\mathbb{R}^d_+.
\end{equation}

 We denote by $\mathcal{A}_{\alpha,c}(\mathbb{R}_+^d)$ the tensor product space
 \[\mathcal{A}_{\alpha,c}(\mathbb{R}_+^d)=\mathcal{A}_{\alpha_1,c}(\mathbb{R}_+)\otimes\ldots\otimes\mathcal{A}_{\alpha_d,c}(\mathbb{R}_+).\]
$\mathcal{A}_{\alpha,c}(\mathbb{R}_+^d)$ is dense in $L^2(\mathbb R^d_+,x^{2\alpha} dx)$. By taking into account \eqref{oper} we define, for every $\phi\in \mathcal{A}_{\alpha,c}(\mathbb{R}_+^d)$,
\[
B_\alpha^{-1/2}(\phi)=h_\alpha\left(\frac{1}{|y|}h_\alpha(\phi)\right)
\]
Note that if $\phi\in \mathcal{A}_{\alpha,c}(\mathbb{R}_+^d)$,
$0\notin \supp(h_\alpha(\phi))$ and then $B_\alpha^{-1/2}\phi\in
S(\mathbb{R}_+^d)$. According to \cite[(5.37), p. 103]{Leb}, we get
\[
R_{\alpha,i} \phi=\partial_{x_i}B_\alpha^{-1/2}(\phi),\,\,\,\phi\in \mathcal{A}_{\alpha,c}(\mathbb{R}_+^d).
\]
$R_{\alpha,i}$ is a Riesz transform associated with $B_\alpha$ in
the sense of Stein \cite{StLP}.

In \cite[Theorem 1.4]{BCC} it was proved that $R_{\alpha,i}$ can be extended from $L^2(\mathbb R^d_+,x^{2\alpha}dx)\cap L^p(\mathbb R^d_+,x^{2\alpha} dx)$ to $L^p(\mathbb R^d_+,x^{2\alpha} dx)$ as an $L^p$-bounded operator, for every $1<p<\infty$, and from $L^1(\mathbb R^d_+,x^{2\alpha} dx)$ into $L^{1,\infty}(\mathbb R^d_+,x^{2\alpha} dx)$. The proof of \cite[Theorem 1.4]{BCC} is very laborious. The underlying set $\mathbb R^d_+$ is divided into many regions, one of them is close to the diagonal of the $i$-th variable and the others far from this diagonal set. Then, the $i$-th Riesz transform is decomposed in some products of Hardy type operators and a local classical Riesz transform. The $L^p$-boundedness properties of $R_{\alpha,i}$ are deduced from the corresponding properties of the operators appearing in the decomposition. This procedure does not allow them to obtain dimension free $L^p$-estimates for $R_{\alpha,i}$ since many iterations of operators are involved.

We define
\[R_\alpha=\left(\sum\limits_{i=1}^d |R_{\alpha,i}|^2\right)^{1/2}.\]
By using Bellman functions technique we prove the following results.

\begin{thm}\label{Riesz}For every $d\in \mathbb N$, {$\alpha\in \mathbb{R}^d_{\geq 0}$}, $1<p<\infty$ and $f\in L^p(\mathbb R^d_+,x^{2\alpha} dx)$, we have
\begin{equation}\label{acotRiesz}
\|R_\alpha f\|_{L^p(\mathbb R^d_+,x^{2\alpha} dx)}\leq 48 (p^*-1)\|f\|_{L^p(\mathbb R^d_+,x^{2\alpha} dx)},
\end{equation}
where $p^*=\max\{p, p/(p-1)\}$.
\end{thm}

For every $k\in \mathbb N$, we define the set
\[C_{\alpha,k}=\{\textrm{compositions of }k \textrm{ operators among } R_{\alpha,1},\dots,R_{\alpha,d}\}.\]

\begin{thm}\label{Rieszcompositions}For every $d,k\in \mathbb N$, {$\alpha\in \mathbb{R}^d_{\geq 0}$}, $1<p<\infty$ and $f\in L^p(\mathbb R^d_+,x^{2\alpha} dx)$
\begin{equation*}
\left\|\left(\sum\limits_{R\in C_{\alpha,k}}|R f|^2\right)^{1/2}\right\|_{L^p(\mathbb R^d_+,x^{2\alpha} dx)}\leq 48^k (p^*-1)^k\|f\|_{L^p(\mathbb R^d_+,x^{2\alpha} dx)}.
\end{equation*}
\end{thm}

It should be mentioned that these theorems are not covered by the general results in \cite{W2}.

From Theorem \ref{Riesz} it is immediate to see that,
for every $i=1,\ldots,d$,
\begin{equation}\label{unid}
\|R_{\alpha,i}f\|_{L^p(\mathbb{R}^d_+,x^{2\alpha}dx)}\le
48(p^*-1)\|f\|_{L^p(\mathbb{R}^d_+,x^{2\alpha}dx)},
\end{equation}
for every $f\in L^p(\mathbb{R}^d_+,x^{2\alpha}dx)$ and $1<p<\infty$.
Dimension free estimates like \eqref{unid} were established in
\cite[Theorem 4.1 (i)]{W1} but there the constants are not
specified.

This article is organized as follows. In Section 2 we prove some
properties related to Bessel--Poisson kernels and integrals. Proofs
of Theorems \ref{Riesz} and \ref{Rieszcompositions} are presented in
Section 3.

Throughout this paper, $c$ and $C$ will denote positive constants that may change from one line to another.

\section{Definitions and properties of Bessel--Poisson operators}

Let $\alpha=(\alpha_1,\dots,\alpha_d)\in {\mathbb R^d_{\geq 0}}$,
$x,y\in \mathbb R^d_+$ and set $\Eins:=(1,1,\ldots, 1)$. From now
on, by $x_j$ and $y_j$ we denote the $j$-th coordinate of $x$ and
$y$, respectively. Let us write symbols for several operations
 \begin{align*}
 xy&=(x_1y_1,\cdots, x_dy_d), \nonumber\\
\alpha-{\amsbb{1}}/2&=\left(\alpha_1-\frac{1}{2},\alpha_2-\frac{1}{2},\ldots, \alpha_d-\frac{1}{2}\right)\nonumber \\ x^{\alpha}&=x_1^{\alpha_1}\cdots x_d^{\alpha_d}=\prod_{j=1}^d x_j^{\alpha_j}.\nonumber
\end{align*}
Given $\nu\in (-1/2,\infty)^d$, we define
\[J_\nu (x)=\prod_{j=1}^d J_{\nu_j}(x_j),  \quad I_\nu(x)= \prod_{j=1}^dI_{\nu_j}(x_j),\]
where, for every $\mu>-1/2$, $J_{\mu}$ and $I_{\mu}$ are the Bessel
function and the modified Bessel function, respectively, of first
kind and order $\mu$.

 Let us define the $d$-dimensional Bessel operator as
\[\Delta_{\alpha,x}=-\sum_{j=1}^d\delta_j^*\delta_j=\sum_{j=1}^dx_j^{-2\alpha_j}\frac{\partial}{\partial x_j}\left(x_j^{2\alpha_j}\frac{\partial}{\partial x_j}\right)=\sum_{j=1}^d\left(\partial_{x_j}^2-\frac{2\alpha_j}{x_j}\partial_{x_j}\right),\]
with
\[\delta_j^*= -x_j^{-2\alpha_j}\frac{\partial}{\partial x_j}(x_j^{2\alpha_j} \cdot),\quad \delta_j=\frac{\partial}{\partial x_j}=:\partial_{x_j},\]
then $B_\alpha=-\Delta_{\alpha,x}.$
Let us remark that the commutator
\[ [\delta_j,\delta_j^*]=\delta_j \delta_j^*-\delta_j^*\delta_j=\frac{2\alpha_j}{x_j^2}\ge 0.\]

In this setting the Hankel transformation $h_\alpha$ plays the same role as the Fourier transformation in the Laplacian context.

We define the heat semigroup $\{T_t^\alpha\}_{t>0}$ and the Poisson semigroup $\{P_t^\alpha\}_{t>0}$ associated with the Bessel operator $B_\alpha$ as follows. For every $f\in L^2(\mathbb{R}_+^d,x^{2\alpha}dx)$,
\[T_t^\alpha  f=h_\alpha (e^{-|\cdot|^2 t}h_\alpha f),\,\,\,t>0,\]
and
\[P_t^\alpha f= h_\alpha (e^{-|\cdot|t}h_\alpha f),\,\,\,t>0.
\]
Therefore, for every $t>0$, we have
 \begin{align}
 T_t^\alpha f(x) &= \int_{\mathbb R_+^d} e^{-|z|^2t} \phi_z^\alpha(x)h_\alpha f(z) z^{2\alpha}dz\nonumber \\ & =\int_{\mathbb R_+^d} \left(\int_{\mathbb R_+^d} e^{-|z|^2t}\phi_z^\alpha(x)\phi_z^\alpha(y)z^{2\alpha}dz\right)f(y)y^{2\alpha}dy\nonumber \\ &=\int_{\mathbb R^d_+} \mathcal{W}_t^\alpha (x,y)f(y)y^{2\alpha}dy,\,\,\, x\in \mathbb{R}_+^d, \nonumber
 \end{align}
 with
 \begin{align}
 \mathcal{W}_t^\alpha(x,y)&=(xy)^{-\alpha+{\amsbb{1}}/2}\int_{\mathbb R^d_+} e^{-|z|^2t}J_{\alpha-\frac{{\amsbb{1}}}{2}}(xz)J_{\alpha-\frac{{\amsbb{1}}}{2}}(yz)z^{\amsbb{1}} dz\nonumber \\
 & =\frac{e^{-\frac{|x|^2+|y|^2}{4t}}}{(2t)^{|\alpha|+d/2}}\left(\frac{xy}{2t}\right)^{-\alpha+{\amsbb{1}}/2}I_{\alpha-\frac{{\amsbb{1}}}{2}}\left(\frac{xy}{2t}\right),\,\,\,x,y\in \mathbb{R}_+^d, \nonumber
 \end{align}
(see \cite[p.~395]{Wat} for the one-dimensional case), where
$\mathcal{W}_t^\alpha(x,y)>0$ for every $t>0$ and every $x,y\in
\mathbb R^d_+$. Thus, $T_t^\alpha $ turns out to be a positive
operator, i.e., for every $0\le f\in
L^2(\mathbb{R}_+^d,x^{2\alpha}dx)$, $T_t^\alpha  f\ge 0.$

By using the subordination formula, we get the following expression for the Poisson semigroup
\begin{align}P_t^\alpha f(x)& =\frac{t}{\sqrt{2\pi}}\int_0^\infty \frac{e^{-\frac{t^2}{4u}}}{u^{3/2}}T_u^\alpha f(x)du\nonumber \\
&=\int_0^\infty \left(\frac{t}{\sqrt{2\pi}}\int_0^\infty \frac{e^{-\frac{t^2}{4u}}}{u^{3/2}}\mathcal{W}_u^\alpha (x,y)du\right)f(y)y^{2\alpha}dy\nonumber \\
&:= \int_0^\infty \mathcal{P}_t^\alpha
(x,y)f(y)y^{2\alpha}dy,\,\,\,x\in \mathbb{R}_+^d\,\,\,
\mathrm{and}\,\,\, t>0,\nonumber
\end{align}
where, again, $\mathcal{P}_t^\alpha (x,y)$ is a positive kernel, so $P_t^\alpha$ is a positive operator as well. In Section \ref{propiedades} we shall give several estimates for the heat and Poisson operators and kernels.

Also, for every $f\in S(\mathbb{R}_+^d)$, we have that
\[\left(\frac{\partial}{\partial t}-B_\alpha\right) T^\alpha_tf(x)=0,\,\,\,x\in \mathbb{R}_+^d\,\,\, \mathrm{and}\,\,\, t>0,\]
and,
\[\mathcal{L}_\alpha P_t^\alpha f(x)=0\,\,\,x\in \mathbb{R}_+^d\,\,\, \mathrm{and}\,\,\, t>0,\]
where
\[\mathcal{L}_\alpha:=\frac{\partial^2}{\partial t^2}-B_\alpha.\]

Let $i=1,\ldots,d$. In order to define the $i$-th conjugate
diffusion and conjugate Poisson semigroups, we shall consider the
conjugate operator of $B_\alpha$, given by $\mathbb{B}_{\alpha,
x}^i=-\amsbb{\Delta}_{\alpha,x}^i$, being
\[\amsbb{\Delta}_{\alpha,x}^i:=-B_{\alpha}-[\delta_i,\delta_i^*]=\Delta_{\alpha,x}-\frac{2\alpha_i}{x_i^2}.\]
In this case we have \[\mathbb{B}_{\alpha, x}^i(-\partial_{x_i}
\phi^\alpha_y(x))=|y|^2 (-\partial_{x_i}
\phi^\alpha_y(x)),\,\,\,x,y\in \mathbb{R}_+^d.\] Then, the diffusion
and the Poisson semigroups associated with this new operator are
defined, respectively, by
\[\mathbb{T}^{\alpha,i}_t g(x)=\int_{\mathbb R^d_+} e^{-|y|^2t}(-\partial_{x_i} \phi_y^\alpha (x)) h_{\alpha+e_i}\left(\frac{ g}{\cdot_i}\right)(y)y^{2\alpha}dy,\,\,\,x\in \mathbb{R}_+^d\,\,\, \mathrm{and}\,\,\, t>0,\]
and
\[\mathbb{P}^{\alpha,i}_t g(x)=\int_{\mathbb R_+^d} e^{-|y|t}(-\partial_{x_i} \phi_y^\alpha (x)) h_{\alpha+e_i} \left(\frac{g}{\cdot_i}\right)(y)y^{2\alpha}dy,\,\,\,x\in \mathbb{R}_+^d\,\,\, \mathrm{and}\,\,\, t>0,\]
with $g\in L^2(\mathbb{R}_+^d,x^{2\alpha}dx)$ and
$\frac{g}{\cdot_i}(z):=\frac{g(z)}{z_i}$, $z=(z_1,\ldots,z_d)\in
\mathbb{R}_+^d$. Thus, for every $g\in S(\mathbb{R}_+^d)$,
\[\left(\frac{\partial}{\partial t}-\amsbb{\Delta}_{\alpha,x}^i\right)\mathbb{T}^t_{\alpha,i} g(x)=0,\,\,\,x\in \mathbb{R}_+^d\,\,\, \mathrm{and}\,\,\, t>0,\] and, if we define \[\mathbb{L}_{\alpha}^i=\frac{\partial^2}{\partial t^2}+\amsbb{\Delta}_{\alpha,x}^i,\] we also have
\[\mathbb{L}_{\alpha}^i \mathbb{P}_t^{\alpha,i} g(x)=0,\,\,\,x\in \mathbb{R}_+^d\,\,\, \mathrm{and}\,\,\, t>0,\] or, equivalently,
\[\mathcal{L}_\alpha\mathbb{P}_t^{\alpha,i} g(x)=\frac{2\alpha_i}{x_i^2}\mathbb{P}_t^{\alpha,i} g(x),\,\,\,x\in \mathbb{R}_+^d\,\,\, \mathrm{and}\,\,\, t>0.\]
From the fact that
\[\frac{\partial}{\partial x_i}\phi_y^\alpha (x)=-x_iy_i^2\phi_y^{\alpha+e_i}(x),\,\,\,x,y\in \mathbb{R}_+^d,\]
we can rewrite the heat and Poisson semigroups by means of Hankel transforms in the following way
\[\mathbb{T}^{\alpha,i}_t g=x_i h_{\alpha+e_i}\left(e^{-|\cdot|^2t}h_{\alpha+e_i}\left(\frac{g}{\cdot_i}\right)\right),\,\,\,t>0,\]
\[\mathbb{P}_t^{\alpha,i} g=x_ih_{\alpha+e_i}\left(e^{-|\cdot| t}h_{\alpha+e_i} \left(\frac{g}{\cdot_i}\right)\right),\,\,\,t>0,\]
for every $g\in L^2(\mathbb{R}_+^d,x^{2\alpha}dx)$.

It is easy to see that
\[\mathbb{T}_{t}^{\alpha,i} g(x)=x_iT_{\alpha+e_i}^t\left(\frac{g}{\cdot_i}\right)(x).\]
Hence
\begin{equation}\label{Z1}\mathbb{P}_t^{\alpha,i}(g)(x)=x_iP_t^{\alpha+e_i}\left(\frac{g}{\cdot_i}\right)(x).
\end{equation}
From these observations we conclude that both
$\mathbb{T}_t^{\alpha,i}$ and $\mathbb{P}_t^{\alpha,i}$ are positive
operators. Besides, since we know that for every $\mu\ge -1/2,$
$I_{\mu+1}(z)< I_\mu(z)$ for all $z>0,$ (see \cite{Nas} and
\cite{So}) then for $\mu\ge 0,$ $u,z>0$ we have
\[uz\mathcal{W}^{\mu+1}_t(u,z)\le \mathcal{W}^{\mu}_t(u,z),
\,\,\,u,z,t>0.\] Hence if $0\le g\in L^2(\mathbb{R}_+^d,
x^{2\alpha}dx),$ then
\[\mathbb{T}_{t}^{\alpha,i} g\le T^\alpha_t g,  \,\,\,t>0, \] and
also
\begin{equation}\label{P<PP}
\mathbb{P}_t^{\alpha,i} g\le P_t^\alpha g,\,\,\,t>0.
\end{equation}

\subsection{Estimates for Bessel--Poisson operators}\label{propiedades}

In this section we establish some properties for the Poisson kernels and the operators $P_t^\alpha$ and $\mathbb{P}_t^{\alpha,i}$ that will be useful in the proof of our main result.

In the sequel, given $x=(x_1,\dots, x_d)\in \mathbb R^d_+$, we will use the notation
\[\widehat{x}^i:=(x_1,\dots, x_{i-1},x_{i+1}, \dots, x_d)\in \mathbb R_+^{d-1}\]
and we will understand $x=(x_i, \widehat{x}^i)$ and
$(x,t)=(x_i,\widehat{x}^i,t)$, $t>0$. Similarly, when
$\alpha=(\alpha_1,\dots, \alpha_d)\in {\mathbb R^d_{\geq 0}}$,
$\widehat{\alpha}^i:=(\alpha_1,\dots, \alpha_{i-1},\alpha_{i+1},
\dots, \alpha_d)\in {\mathbb R^{d-1}_{\geq 0}}$ and
$\alpha=(\alpha_i, \widehat{\alpha}^i)$.

\begin{prop}\label{estimates}Let {$\alpha\in \mathbb R^d_{\geq 0}$} and $f\in C_c^\infty(\mathbb R^d_+)$. Then,
 \begin{enumerate}[label=\textnormal{(P.\arabic*)}]
 \item \label{P-tinfty}$\lim_{t\to +\infty} P_t^\alpha f(x)=0$, for every $x\in \mathbb R_+^d$,
 \item \label{tdP-t0}$\lim_{t\to 0^+}t\partial_t P_t^\alpha f(x)=0$, for every $x\in \mathbb R_+^d$,
 \item \label{tdP-tinfty}$\lim_{t\to +\infty}t\partial_t P_t^\alpha f(x)=0$, for every $x\in \mathbb R_+^d$,
\item  \label{dP-x0} $\lim_{x_i\to 0^+}\partial_{x_i} P_t^\alpha f({ x_i, \widehat{x}^i})=0$, for every $i=1,\ldots,d$, $t>0,\; \widehat{x}^i\in \mathbb R_+^{d-1}.$
\item \label{Des1} For every $x,y\in \mathbb{R}^d_+$ and $t>0$,
\[
0\le P_t^\alpha(x,y)\le
C\frac{t}{(t^2+|x-y|^2)^{|\alpha|+\frac{d+1}{2}}}.
\]
\item \label{Des2} For every $i=1,\ldots,d$,
\[
| \partial_{x_i}P_t^\alpha(x,y)|\le
C\frac{t(x_i+y_i)}{(t^2+|x-y|^2)^{|\alpha|+\frac{d+3}{2}}},\,\,\,x,y\in
\mathbb{R}_+^d\,\,\,and\,\,\,t>0.
\]
\end{enumerate}\end{prop}

\begin{proof} Since $h_\alpha(f)\in S(\mathbb{R}^d_+)$, by taking into account that, for every $\mu\ge -1/2 $, the function $z^{-\mu}J_\mu(z)$ is bounded on $(0,\infty)$, we can differentiate under the integral sign to get
 \begin{align*}
 t\partial_t P_t^\alpha f(x)&= -\int_{\mathbb R_+^d} |y|t\ e^{-|y|t} \phi_y^\alpha(x)h_\alpha (f)(y) y^{2\alpha}dy\\
&=- h_\alpha(|\cdot| t\  e^{-|\cdot|\ t}h_\alpha(f)(\cdot))(x),
\,\,\,x\in \mathbb{R}_+^d\,\,\, and\,\,\, t>0,
\end{align*}
and, since
$\frac{d}{dz}(z^{-\mu}J_\mu(z))=-z^{-\mu}J_{\mu+1}(z)$, $z>0$ and
$\mu\ge -1/2$ (\cite[(5.3.5), p. 103]{Leb}), for every
$i=1,\ldots,d$, $x\in \mathbb{R}^d_+$ and $t>0$,
$$
\partial_{x_i}P_t^\alpha
f(x)=-x_i\int_{\mathbb{R}^d_+}e^{-|y|t}\phi_{\widehat{y}^i}^{\widehat{\alpha}^i}(\widehat{x}^i)(x_iy_i)^{-\alpha_i-1/2}J_{\alpha_i+1/2}(x_iy_i)y_i^2h_\alpha(f)(y)y^{2\alpha}dy.
$$

By using the Dominated Convergence Theorem we can prove properties
\ref{P-tinfty}-{\ref{dP-x0}.}

Let us prove \ref{Des1}. Recall that, for $\mu\ge -1/2$, the modified
Bessel functions of first kind $I_\mu$ verify
\begin{equation}\label{cotaInu-1/2}
I_{\mu} (z)\leq Cz^{\mu},\,\,\,z\in (0,1),
\end{equation}
and
\begin{equation}\label{InuOgrande}
I_\mu(z)=\frac{e^z}{\sqrt{2\pi
z}}\left(1+O\left(\frac{1}{z}\right)\right),\,\,\,z\ge 1
\end{equation}
(see \cite[(5.11.8), p. 123]{Leb}).

By using the above one-dimensional estimates \eqref{cotaInu-1/2} and
\eqref{InuOgrande} on each variable, we have
\begin{equation}\label{heatest}
\mathcal{W}_t^{\alpha}(x,y)\leq
C_\alpha\frac{e^{-\frac{|x-y|^2}{4t}}}{(2t)^{|\alpha|+\frac{d}{2}}},
\quad x,y \in \mathbb R^{d}_+\,\,\,and \,\,\,t>0.
\end{equation}
Then, the subordination formula leads to
\begin{align*}
\mathcal{P}_t^\alpha (x,y)&\le \frac{t}{\sqrt{2\pi}}\int_0^\infty
\frac{e^{-\frac{t^2}{4u}}}{u^{3/2}}|\mathcal{W}_u^\alpha
(x,y)|du\\
&\leq C t
\int_0^{\infty}\frac{e^{-\frac{t^2+|x-y|^2}{8u}}}{u^{|\alpha|+\frac{d+3}{2}}}du
\\ & \leq C
\frac{t}{(t^2+|x-y|^2)^{|\alpha|+\frac{d+1}{2}}},\,\,\,x,y\in
\mathbb{R}_d^+\,\,\,and\,\,\,t>0.
\end{align*}

Let $i=1,\ldots,d$. In order to prove \ref{Des2} we need estimates
for $\partial_{x_i} \mathcal{W}_t^{\alpha_i}(x_i,y_i)$, $x_i,y_i>0$,
and for
$\mathcal{W}_t^{\widehat{\alpha}^i}(\widehat{x}^i,\widehat{y}^i)$,
$\widehat{x}^i,\widehat{y}^i\in \mathbb R^{d-1}_+$, $i=1,\dots,d$,
with $t>0$.

Similarly to \eqref{heatest} we get
\begin{equation}\label{heat_i}
\mathcal{W}_t^{\widehat{\alpha}^i}(\widehat{x}^i,\widehat{y}^i)\leq
C_\alpha\frac{e^{-\frac{|\widehat{x}^i-\widehat{y}^i|^2}{4t}}}{t^{|\widehat{\alpha}^i|+\frac{d-1}{2}}},
\quad \widehat{x}^i,\widehat{y}^i\in \mathbb
R^{d-1}_+\,\,\,and\,\,\,t>0.
\end{equation}
Since
$\frac{d}{dz}[z^{-\mu}I_\mu(z)]=z^{-\mu}I_{\mu+1}(z)$ for every
$z>0$ and $\mu\ge -1/2$ (\cite[(5.7.9), p. 110]{Leb}), we have
\begin{align}\partial_{x_i} \mathcal{W}_t^\alpha(x,y)&=  \partial_{x_i}\mathcal{W}_t^{\alpha_i}(x_i,y_i) \mathcal{W}_t^{\widehat{\alpha}^i}(\widehat{x}^i,\widehat{y}^i)\nonumber \\
&=
\frac{e^{-\frac{x_i^2+y_i^2}{4t}}}{(2t)^{\alpha_i+1/2}}\left(\frac{x_iy_i}{2t}\right)^{-\alpha_i+\frac{{1}}{2}}\left[-\frac{x_i}{2t}I_{\alpha_i-\frac{{1}}{2}}\left(\frac{x_iy_i}{2t}\right)+\frac{y_i}{2t}I_{\alpha_i+\frac{{1}}{2}}\left(\frac{x_iy_i}{2t}\right)\right]
\nonumber \\ & \quad
\times\mathcal{W}_t^{\widehat{\alpha}^i}(\widehat{x}^i,\widehat{y}^i)
.\nonumber
\end{align}
By using \eqref{cotaInu-1/2} we get, for every $t,x_i,y_i>0$, being
$x_iy_i<2t$,
\begin{align}
|\partial_{x_i} \mathcal{W}_t^{\alpha_i}(x_i,y_i)| &\leq C
\frac{e^{-\frac{x_i^2+y_i^2}{4t}}}{(2t)^{\alpha_i+1/2}}\left(\frac{x_iy_i}{2t}\right)^{-\alpha_i}\left[\frac{x_i}{2t}\left(\frac{x_iy_i}{2t}\right)^{\alpha_i}+\frac{y_i}{2t}\left(\frac{x_iy_i}{2t}\right)^{\alpha_i+1}\right]e^{\frac{x_i
y_i}{2t}}\nonumber\\ & \leq
C\frac{x_i+y_i}{t}\frac{e^{-\frac{x_i^2+y_i^2}{4t}}}{t^{\alpha_i+1/2}}.\nonumber
\end{align}
Now, if $t,x_i,y_i>0$ with $x_iy_i>2t$, using
\eqref{InuOgrande} we have
\begin{align}
|\partial_{x_i} \mathcal{W}_t^{\alpha_i}(x_i,y_i)|&=\left|\frac{e^{-\frac{x_i^2+y_i^2}{4t}}}{(2t)^{\alpha_i+1/2}}\left(\frac{x_iy_i}{2t}\right)^{-\alpha_i}\frac{e^{\frac{x_iy_i}{2t}}}{\sqrt{2\pi}}\right.\nonumber\\
&\quad \left.\times\left[-\frac{x_i}{2t}\left( 1+O\left(\frac{2t}{x_iy_i}\right)\right)+\frac{y_i}{2t}\left( 1+O\left(\frac{2t}{x_iy_i}\right)\right)\right]\right|\nonumber\\
&\leq
C\frac{e^{-\frac{(x_i-y_i)^2}{4t}}}{t^{\alpha_i+1/2}}\left(\frac{x_iy_i}{t}\right)^{-\alpha_i}\frac{x_i+y_i}{t}\nonumber
\leq
C\frac{x_i+y_i}{t^{1/2}}\frac{e^{-\frac{(x_i-y_i)^2}{4t}}}{t^{\alpha_i+1}}.\nonumber
\end{align}

Therefore, combining the above estimate and \eqref{heat_i}, we deduce
\begin{equation}\label{derivheatkernel}
|\partial_{x_i} \mathcal{W}_t^{\alpha}(x,y)|\leq C_\alpha
\frac{x_i+y_i}{t^{1/2}}\frac{e^{-\frac{|x-y|^2}{8t}}}{t^{|\alpha|+\frac{d+1}{2}}}.
\end{equation}
From  \eqref{derivheatkernel}, we obtain, for every $x,y\in
\mathbb{R}_+^d$ and $t>0$,
\begin{align*}
|\partial_{x_i} \mathcal{P}_t^\alpha (x,y)|&\le
\frac{t}{\sqrt{2\pi}}\int_0^\infty
\frac{e^{-\frac{t^2}{4u}}}{u^{3/2}}|\partial_{x_i}\mathcal{W}_u^\alpha
(x,y)|du\nonumber\\
&\leq C t
(x_i+y_i)\int_0^{\infty}\frac{e^{-\frac{t^2+|x-y|^2}{8u}}}{u^{|\alpha|+\frac{d+1}{2}}}\frac{du}{u^2}\nonumber
\\ & \leq C
t\frac{x_i+y_i}{(t^2+|x-y|^2)^{|\alpha|+\frac{d+3}{2}}}.\qedhere
\end{align*}
\end{proof}

\begin{prop}Let {$\alpha\in \mathbb R^d_{\geq 0}$} and $g\in C^\infty_c(\mathbb R^d_+)$. Then, for every $i=1,\dots, d$,
\begin{enumerate}[label=\textnormal{($\mathbb{P}$.\arabic*)}]
\item \label{PP-t0} $\lim_{t\to +\infty} \mathbb{P}_t^{\alpha,i} { g}(x)=0$, for every  $x\in \mathbb R_+^d$,
 \item \label{tdPP-t0}$\lim_{t\to 0^+}t\partial_t \mathbb{P}_t^{\alpha,i}g(x)=0$, for every  $x\in\mathbb R_+^d.$
 \item \label{tdPP-tinfty}$\lim_{t\to +\infty}t\partial_t \mathbb{P}_t^{\alpha,i}g(x)=0$, for every  $x\in \mathbb R_+^d.$
 \end{enumerate}
\end{prop}

\begin{proof} We can proceed as in the proof of the corresponding items of Proposition \ref{estimates}.
\end{proof}

\section{Proofs of Theorems \ref{Riesz} and \ref{Rieszcompositions}}

In this section we present the proofs of Theorems
\ref{Riesz} and \ref{Rieszcompositions}. First, we shall prove some
auxiliary results.

\begin{lem}\label{Ralfaifg} Let {$\alpha\in \mathbb R^d_{\geq 0}$} and $f,g \in L^2(\mathbb R^d_+, x^{2\alpha}dx)$. Then,
\begin{equation}\label{Lemma}
\int_{\mathbb R_+^d} R_{\alpha,i} f(x) g(x)x^{2\alpha}dx=-4\int_0^\infty \int_{\mathbb{R}^d_+} \partial_{x_i} P^\alpha_t (f)(x)\partial_t \mathbb {P}^{\alpha,i}_t (g)(x)x^{2\alpha}dx t dt.
\end{equation}
for every $i=1,\dots,d.$
\end{lem}

\begin{proof}
Let $i=1,\ldots,d$. According to \cite[(5.3.7), p. 103]{Leb} we have
that
\begin{equation}\label{F1}
\partial_{x_i}P_t^\alpha(f)(x)=-x_ih_{\alpha+e_i}\left(e^{-|y|t}h_\alpha(f)(y)\right)(x),\,\,\,x\in \mathbb{R}_+^d\,\,\,and\,\,\,t>0.
\end{equation}
Also, we get
\begin{equation}\label{F2}
\partial_{t}\mathbb{P}_t^{\alpha,i}(g)(x)=-x_ih_{\alpha+e_i}\left(|y|e^{-|y|t}h_{\alpha+e_i}\left(\frac{g}{._i}\right)(y)\right)(x),\,\,\,x\in \mathbb{R}_+^d\,\,\,and\,\,\,t>0.
\end{equation}
Note that differentiations under the integral sign to obtain
\eqref{F1} and \eqref{F2} are justified because the function
$z^{-\mu}J_\mu(z)$ is bounded on $(0,\infty)$ for every $\mu\ge-1/2$
(\cite[(5.16.1), p. 134]{Leb}).

From \eqref{F1} and \eqref{F2} we can write
\begin{align*}
\int_{0}^\infty &\int_{\mathbb{R}^d_+}\partial_{x_i} P^\alpha_t
(f)(x)\partial_t \mathbb {P}^{\alpha,i}_t (g)(x)x^{2\alpha}dx t dt
\\&= \int_0^\infty \int_{\mathbb R_+^d} x_i\,
h_{\alpha+e_i}(e^{-|\cdot| t}h_\alpha f)(x)
x_i\, h_{\alpha+e_i}\left(|\cdot| e^{-|\cdot| t}h_{\alpha+e_i}\left(\frac{g}{\cdot_i}\right) \right)(x)\, x^{2\alpha}dx tdt\\
&=\int_0^\infty \int_{\mathbb R_+^d}  h_{\alpha+e_i}(e^{-|\cdot| t}h_\alpha f)(x)
h_{\alpha+e_i}\left(|\cdot| e^{-|\cdot| t}h_{\alpha+e_i}\left(\frac{g}{\cdot_i}\right) \right)(x)x^{2(\alpha+e_i)}dxt dt.
\end{align*}
Since the Hankel transform $h_{\alpha+e_i}$ is self-adjoint on
$L^2(\mathbb R^d_+,y^{2(\alpha+e_i)}dy)$ and coincides with its
inverse, we obtain
\begin{align*}\int_{0}^\infty \int_{\mathbb{R}^d_+} &\partial_{x_i} P^\alpha_t (f)(x)\partial_t \mathbb {P}^{\alpha,i}_t (g)(x)x^{2\alpha}dx t dt \\&=\int_0^\infty \int_{\mathbb R_+^d}|y|  e^{- 2|y|t}h_\alpha (f)(y)h_{\alpha+e_i}\left(\frac{g(\cdot)}{\cdot_i}\right)(y) y^{2(\alpha+e_i)}dy\;t\, dt.
\end{align*}
Notice that we can interchange the order of integration above since $h_\alpha$ and $h_{\alpha+e_i}$ are isometries in $L^2(\mathbb R^d_+,y^{2\alpha}dy)$ and $L^2(\mathbb R^d_+,y^{2(\alpha+e_i)}dy)$, respectively. Indeed, we have that
\begin{align*}
\int_0^\infty &\int_{\mathbb R_+^d} t|y| e^{-2|y| t}|h_\alpha (f)(y)|\left|h_{\alpha+e_i}\left(\frac{g(\cdot)}{\cdot_i}\right)(y)\right| y^{2(\alpha+e_i)}dy\, dt\\
&\leq \frac{1}{4} \int_{\mathbb R^d_+} \frac{1}{|y|}|h_\alpha (f)(y)|\left|h_{\alpha+e_i}\left(\frac{g(\cdot)}{\cdot_i}\right)(y)\right| y^{2(\alpha+e_i)}dy\\
&\leq \frac{1}{4}\left(\int_{\mathbb R^d_+} \left(\frac{y_i}{|y|}|h_\alpha (f)(y)|\right)^2 y^{2\alpha}dy\right)^{1/2}\left(\int_{\mathbb R^d_+} \left|y_ih_{\alpha+e_i}\left(\frac{g(\cdot)}{\cdot_i}\right)(y)\right|^2 y^{2\alpha}dy\right)^{1/2}\\
&\leq \frac{1}{4} \|h_\alpha (f)\|_{L^2(\mathbb R^d_+, y^{2\alpha}dy)} \left\|h_{\alpha+e_i}\left(\frac{g(\cdot)}{\cdot_i}\right)\right\|_{L^2(\mathbb R^d_+, y^{2(\alpha+e_i)}dy)} \\
&\leq \frac{1}{4} \| f\|_{L^2(\mathbb R^d_+, y^{2\alpha}dy)} \|g\|_{L^2(\mathbb R^d_+, y^{2\alpha}dy)} <\infty.
\end{align*}

Thus, we get
\begin{align*}
\int_{0}^\infty \int_{\mathbb{R}^d_+} &\partial_{x_i} P^\alpha_t (f)(x)\partial_t \mathbb {P}^{\alpha,i}_t (g)(x)x^{2\alpha}dx t dt\\
&= \int_{\mathbb R_+^d} \left(\int_0^\infty |y|^2e^{-2|y|t}t dt\right)\frac{ h_\alpha f(y)}{|y|}h_{\alpha+e_i}\left(\frac{g}{\cdot_i}\right)(y) y^{2(\alpha+e_i)}dy \\
&= \frac{1}{4}\int_{\mathbb R_+^d} h_{\alpha+e_i}\left( \frac{ h_\alpha f}{|\cdot|}\right)(y) \frac{g(y)}{y_i} y^{2(\alpha+e_i)}dy \\
& =\frac{1}{4}\int_{\mathbb R_+^d} y_i h_{\alpha+e_i}\left( \frac{ h_\alpha f}{|\cdot|}\right)(y) g(y) y^{2\alpha}dy \\
&= -\frac{1}{4}\int_{\mathbb R_+^d} R_{\alpha,i} f(y) \, g(y) \,
y^{2\alpha}dy.\qedhere
 \end{align*}
\end{proof}

\begin{rem} Let {$\alpha\in \mathbb R^d_{\geq 0}$}. We define Littlewood-Paley functions associated with Bessel-Poisson semigroups as
follows: for every $f\in L^p(\mathbb R^d_+, x^{2\alpha}dx)$,
$1<p<\infty$ {and $i=1,\ldots,d$,}
\[
G_{\alpha,0}^{i}(f)(x)=\left(\int_0^\infty
t|\partial_t\mathbb{P}_t^{\alpha,{i}}(f)(x)|^2dt\right)^{1/2},\,\,\,x\in
\mathbb{R}_+^d,
\]
{and}
\[
G_{\alpha,i}(f)(x)=\left(\int_0^\infty
t|\partial_{x_i}P_t^\alpha(f)(x)|^2dt\right)^{1/2},\,\,\,x\in
\mathbb{R}_+^d.
\]
The heat semigroup $\{T_t^\alpha\}_{t>0}$ is a diffusion semigroup
in the Stein's sense (\cite[p. 65]{StLP}). Then,
$\{P_t^\alpha\}_{t>0}$ and
$\{\mathbb{P}_t^{\alpha,i}\_{t>0}$ are also Stein
diffusion semigroups. According to \cite[Theorem 10]{StLP} (see also
\cite{Meda}), for every $1<p<\infty$ {and $i=1,\ldots,d$,}, there
exists $C_p>0$ depending only on $p$, such that
\[
\|G_{\alpha,0}^{i}(f)\|_{p}\le C_p\|f\|_p, \,\,\,f\in L^p(\mathbb
R^d_+, x^{2\alpha}dx).
\]
By proceeding as in the proof of \cite[Theorem 6]{No} we can see
that for every $1<p<\infty$ and $i=1,\ldots,d$, there exists $C_p>0$
depending only on $p$, such that
\[
\|G_{\alpha,i}(f)\|_{p}\le C_p\|f\|_p, \,\,\,f\in L^p(\mathbb R^d_+,
x^{2\alpha}dx).
\]
By using Lemma \ref{Ralfaifg} and H\"older's inequality we can prove
that, for every $1<p<\infty$ and $i=1,\ldots,d$, there exists
$C_p>0$ depending only on $p$, such that
\[
\|R_{\alpha,i} f\|\le C_p\|f\|_p,\,\,\,f\in L^p(\mathbb R^d_+,
x^{2\alpha}dx).
\]
Our next argument, that uses Bellman functions, allows us to improve
the last boundedness inequality getting an explicit expression for
the constant $C_p$.
\end{rem}

\vspace{3mm}

A fundamental ingredient for the proof of our main result is the Bellman function. A function introduced by Nazarov and Treil (\cite{NT}) and some modifications of it have been used in different settings (\cite{CD}, \cite{DV1}, \cite{DV2}, \cite{DV3}, \cite{W1}, and \cite{W2}).

Assume that $p \geq 2$. Note that $p^*=p$. We denote as usual
$p'=p/(p-1)$ and we define $\gamma(p)= \frac{p'(p'-1)}{8}$. As in
\cite{W2} we consider the function $\beta_p$ defined by
\[
\beta_p(s_1,s_2)= s_1^p+s_2^p + \gamma\left\{\begin{array}{lr} s_1^2s_2^{2-p'}, & s_1^p \leq s_2^q,\\
\frac{2}{p}s_1^p + \left(\frac{2}{p'}-1 \right)s_2^{p'},& s_1^p \geq
s_2^q \end{array}\right.,\quad  s_1,s_2 \geq 0.
\]
For $m=(m_1,m_2)\in \mathbb N^2$, the Nazarov-Treil Bellman function $B_p=B_{p,m}:\mathbb{R}^{m_1}\times\mathbb R^{m_2} \to [0,\infty)$ is defined by
\[B_p(\zeta, \eta)=\frac{1}{2}\beta_p(|\zeta|,|\eta|),\]
with $\zeta\in \mathbb R^{m_1}$ and $\eta\in\mathbb R^{m_2}.$ Let us
remark that $B_p\in C^1(\mathbb{R}^{m_1+m_2})$ and $B_p\in
C^2(\mathbb{R}^{m_1+m_2}\setminus \{(\zeta, \eta): \eta=0\
\text{or}\ |\zeta|^p=|\eta|^q\}).$ In order to get rid of this lack
of smoothness we convolve this function with a smooth
non-negative radial approximation to the identity
$\{\psi_\kappa\}$ where $\kappa>0.$ We call it
\[B_{\kappa,p}=B_p\ast \psi_\kappa.\]
Since both $B_p$ and $\psi_\kappa$ are bi-radial then $B_{\kappa,p}$
is bi-radial and therefore there exists
$\beta_{\kappa,p}:[0,\infty)^2\to [0,\infty) $ such that
$B_{\kappa,p} (\zeta, \eta)=\frac{1}{2}\beta_{\kappa,p} (|\zeta|,
|\eta|).$ The properties of the functions $\beta_{\kappa,p}$ and
$B_{\kappa,p}$ that we need are stated in \cite[Proposition 4]{W2}.
We list them below for the sake of completeness.

\begin{prop}[\cite{W2}]\label{bellman}
Let $\kappa \in (0,1).$ Then for $r,s>0,$ we have
\begin{enumerate}[label=\textnormal{($B$.\arabic*)}]
\item \label{cotaBk} $0\le \beta_{\kappa,p} (r,s)\le (1+\gamma(p))((r+\kappa)^p+(s+\kappa)^{p'}),$
\item \label{partialBk}$0\le \partial_r \beta_{\kappa,p} (r,s)\le C_p\, \max\{(r+\kappa)^{p-1}, s+\kappa\}$ and $0\le \partial_s \beta_{\kappa,p} (r,s)\le C_p(s+\kappa)^{p'-1},$ with $C_p$ being a positive constant.
\end{enumerate}
Moreover $B_{\kappa,p} \in C^\infty (\mathbb{R}^{m_1+m_2}),$ and for any $(\zeta,\eta)\in \mathbb{R}^{m_1+m_2}$ there exists a positive $\tau_\kappa=\tau_\kappa(|\zeta|, |\eta|)$ such that for $\omega=(\omega_1, \omega_2)\in \mathbb{R}^{m_1+m_2}$ we have
\begin{enumerate}[label=\textnormal{($B$.\arabic*)}]\addtocounter{enumi}{2}
\item \label{hessBk} $\omega\text{Hess}(B_{\kappa,p})(\zeta,\eta)\omega^T \ge \frac{\gamma(p)}{2}(\tau_\kappa |\omega_1|^2+\tau_\kappa^{-1}|\omega_2|^2),$ where $\text{Hess}(B_{\kappa,p})$ denotes the Hessian matrix of $B_{\kappa,p}$.
\end{enumerate}
\end{prop}

Suppose now that { $f,g_i\in C_c^\infty(\mathbb R^d_+)$ and $g_i\ge
0, i=1, \dots, d$. Let $g=(g_1,\dots, g_d).$} We define the function
{$b_{\kappa, p}(x,t)$ as follows
 \[
 b_{\kappa,p} (x,t)= B_{\kappa,p}(u(x,t)),\;\; x\in \mathbb R^d_+\;\; \mathrm{and}\;\; t>0,
 \]
 where $u(x,t):=(P_t^\alpha(f)(x), \mathbb P_t^{\alpha}(g)(x))$ and $\mathbb P_t^\alpha g(x):=(\mathbb P_t^{\alpha,1}(g_1)(x), \dots  \mathbb P_t^{\alpha,d}(g_d)(x))$. Note that in this case $m_1=1$ and $m_2=d.$}

The following lemma allows us to obtain a crucial connection between
the function {$b_{\kappa, p}$} and the right hand side of
\eqref{Lemma}.

\begin{lem}Let {$\alpha \in \mathbb R^d_{\geq 0}$}, and $\kappa\in (0,1)$.
 Assume that $f,g_i \in C_c^\infty(\mathbb R^d_+)$,
$g_i \geq 0$, and $i=1,\ldots,d$. Then, we have the
following pointwise equality {\begin{align*}
\mathcal L_\alpha b_{\kappa, p}(x,t)&=(\partial_t u)\text{Hess}(B_{\kappa,p}(u(x,t)))(\partial_t u)^T\\
&\quad +\sum_{j=1}^d(\partial_{x_j} u)\text{Hess}(B_{\kappa,p}(u(x,t)))(\partial_{x_j} u)^T \\
&\quad +\sum_{j=1}^d\partial_{\eta_j}
B_{\kappa,p}(u(x,t))\frac{2\alpha_j}{x_j^2}\mathbb{P}_t^{\alpha,j}({g_j})(x),
\end{align*}}
for every  $x\in \mathbb R^d_+$ and $t>0$.

Moreover,
\begin{equation*}
\mathcal L_\alpha b_{\kappa, p}(x,t)\geq \gamma(p) |
P_t^\alpha(f)(x)|_*{| \mathbb{P}_t^{\alpha}(g)(x)|_*}\,\,\,,x\in
\mathbb{R}_+^d\,\,\,and\,\,\,t>0.
\end{equation*}
Here (see \cite[(3.4)]{W2}) we define
\[
| P_t^\alpha(f)(x)|_*:=
\left(|\partial_tP_t^\alpha(f)(x)|^2+\sum_{j=1}^d|\partial_{x_j}P_t^\alpha(f)(x)|^2\right)^{1/2},\,\,\,x\in
\mathbb{R}_+^d\,\,\,and\,\,\,t>0,
\]
and $|
\mathbb{P}_t^{\alpha}(g)(x)|_*:=\left(\sum_{i=1}^d|\mathbb{P}_t^{\alpha,i}(g_i)(x)|_*^2\right)^{1/2}$, $x\in \mathbb{R}^d_+$ and $t>0$.
\end{lem}

\begin{proof}
The first part of this lemma can be proved following the same lines
as in \cite[Proposition 5]{W2}. We use that $\mathcal{L}_\alpha
P_t^\alpha(f)(x)=0$ and, for every $i=1,\ldots,d$,
{\begin{equation}
\label{wrobel}\mathcal{L}_\alpha\mathbb{P}_t^{\alpha,i}(g_i)(x)=\frac{2\alpha_i}{x_i^2}\mathbb{P}_t^{\alpha,i}(g_i)(x),
\end{equation}}
$x\in \mathbb{R}_+^d$ and $t>0$.

{By applying now the Bellman function properties \ref{hessBk}  on
each term involving the Hessian, \ref{partialBk} and taking into
account that \eqref{wrobel} is non-negative (since the semigroup and
the function are non-negative)}, we get, for every $\kappa\in
(0,1)$,
\begin{align*}
\mathcal L_\alpha { b_{\kappa, p}(x,t)} &{ \ge
\frac{\gamma(p)}{2}\left[\tau_\kappa |\partial_tP_t^\alpha
f(x)|^2+\tau_\kappa^{-1}\sum_{i=1}^d
|\partial_t\mathbb{P}_t^{\alpha,i}g_i(x)|^2 \right.}\\ &{ \left.
\quad + \sum_{j=1}^d \left(\tau_\kappa |\partial_{x_j}P_t^\alpha
f(x)|^2+\tau_\kappa^{-1}\sum_{i=1}^d
|\partial_{x_j}\mathbb{P}_t^{\alpha,i}g_i(x)|^2 \right)\right]}\\ &
\ge \frac{\gamma(p)}{2}\left(\tau_\kappa |
P_t^\alpha(f)(x)|_*^2+\tau_\kappa^{-1}|
\mathbb{P}_t^{\alpha}(g)(x)|_*^2\right), \,\,\,x\in
\mathbb{R}_+^d\,\,\,and\,\,\,t>0,
\end{align*}
being $\tau_\kappa=\tau_\kappa(x,t)>0$.

Then, since  $as+b/s\geq 2\sqrt{ab}$ for every $s,a,b\in \mathbb
R_+$, we get \begin{align*} \mathcal L_\alpha b_{\kappa,
p}(x,t)&\ge\gamma(p)| P_t^\alpha(f)(x)|_*|
\mathbb{P}_t^{\alpha}(g)(x)|_*.\qedhere
\end{align*}
\end{proof}

 {In the sequel, we shall use the following notation: given $F=(F_1, \dots, F_d)$ with $F_i$, $i=1,\dots,d$, defined on $\mathbb R^d_+$, we denote as usual
    \[|F|=\left(\sum\limits_{i=1}^d |F_i|^2\right)^{1/2}\]
and, for $1< p< \infty$, we will write
\[\left\VERT F\right\VERT _{p}:=\left\|\left(\sum\limits_{i=1}^d |F_i|^2\right)^{1/2}\right\|_{p}.\]}

We are now in position to prove \eqref{acotRiesz}. Let $f,g_i\in
C_c^\infty(\mathbb{R}_+^d)$, $g_i\ge 0$, $i=1,\ldots,d$. We deduce
from Lemma \ref{Ralfaifg} that
\begin{align*}
\left| \sum_{i=1}^d\right.&\left.\int_{\mathbb{R}_+^d}R_{\alpha,i}(f)(x)g_i(x)x^{2\alpha}dx\right|\\ & =4\left|\int_0^\infty\int_{\mathbb{R}_+^d}t\sum_{i=1}^d\partial_{x_i}P_t^\alpha(f)(x)\partial_t\mathbb{P}_t^{\alpha,i}(g_i)(x)x^{2\alpha}dxdt\right|\\
&\le
4\int_0^\infty\int_{\mathbb{R}_+^d}t\left(\sum_{i=1}^d|\partial_{x_i}P_t^\alpha(f)(x)|^2\right)^{1/2}\left(\sum_{i=1}^d|\partial_t\mathbb{P}_t^{\alpha,i}(g_i)(x)|^2\right)^{1/2}x^{2\alpha}dxdt
\\ &\le 4 \int_0^\infty \int_{\mathbb R_+^d} | P_t^\alpha
f(x)|_*| \mathbb{P}_t^{\alpha} g(x)|_*x^{2\alpha}dxt dt
\\& \le 4\lim_{\epsilon\to
0^+}\lim_{m\to\infty}\int_0^\infty \int_{[\frac{1}{m},m]^d}|
P_t^\alpha f(x)|_*| \mathbb{P}_t^{\alpha}
g(x)|_*x^{2\alpha}dxe^{-\epsilon t}t dt
\\& \le \frac{4}{\gamma(p)}\limsup_{\epsilon\to
0^+}\limsup_{m\to\infty}\int_0^\infty \int_{[\frac{1}{m},m]^d}
\mathcal{L}_\alpha b_{\kappa_m, p}(x,t)x^{2\alpha}dx e^{-\epsilon
t}tdt,
\end{align*}
provided that $\kappa_m\in (0,1)$, for every $m\in \mathbb{N}$. We
shall prove now that the right-hand side of the above inequality is
bounded by a constant times $\|f\|_p^p+{\left\VERT g\right\VERT
_{p'}^{p'}}$, with constant independent of dimension.

Let $\epsilon>0$ and $m\in \mathbb N$. We take $\kappa_m\in (0,1)$ such that
\[\kappa_m^{p'-1/2}m^{2|\alpha|+d}\le 1,\] which implies, since $p\ge 2$, $\kappa_m^{p-1/2}m^{2|\alpha|+d}\le 1$ and $\lim_{m\rightarrow \infty }\kappa_m=0$. Let $0<a<1<b$ be
given in such a way that $\text{supp}(f)$ and {$\text{supp}(g_i),
i=1\dots, d$} are contained in $[a,b]^d.$  Recalling the definition
of $\mathcal L_\alpha$, we have to estimate
\[I_{m,\epsilon}:=\int_0^\infty \int_{[\frac{1}{m},m]^d} \partial_t^2 { b_{\kappa_m,  p}}(x,t)x^{2\alpha}dx e^{-\epsilon t}tdt,\]
and, for every $\ell=1, \dots, d$,
\[J_{m,\epsilon,\ell}:=\int_0^\infty \int_{[\frac{1}{m},m]^d} x_\ell^{-2\alpha_\ell}\partial_{x_\ell} (x_\ell^{2\alpha_\ell} \partial_{x_\ell} { b_{\kappa_m,  p}}(x,t))x^{2\alpha}dx e^{-\epsilon t}tdt.\]

{\bf Estimation of $I_{m,\epsilon}$:} By using integration by parts we get
\begin{align*}
I_{m,\epsilon}&=\int_0^\infty \partial_t \left(\int_{[\frac{1}{m},m]^d } \partial_t { b_{\kappa_m,  p}}(x,t)x^{2\alpha}dx\, \right)e^{-\epsilon t}tdt \\
&=\lim_{T\to \infty} T e^{-\epsilon T}\int_{[\frac{1}{m},m]^d }\partial_t { b_{\kappa_m,  p}}(x,T)x^{2\alpha}dx \\
&\quad - \lim_{t\to 0^+}te^{-\epsilon t}\int_{[\frac{1}{m},m]^d } \partial_t { b_{\kappa_m,  p}}(x,t)x^{2\alpha}dx   \\
&\quad -\int_0^\infty \int_{[\frac{1}{m},m]^d } \partial_t {
b_{\kappa_m,  p}}(x,t)x^{2\alpha}dx \, e^{-\epsilon t}(1-\epsilon
t)dt.
\end{align*}
The third term can be rewritten in the following way
\begin{align*}
-\int_0^\infty \int_{[\frac{1}{m},m]^d } \partial_t { b_{\kappa_m,  p}}&(x,t) x^{2\alpha}dx \, e^{-\epsilon t}(1-\epsilon t)dt\\
&=-\int_0^\infty \partial_t \left(\int_{[\frac{1}{m},m]^d }  { b_{\kappa_m,  p}}(x,t)x^{2\alpha}dx e^{-\epsilon t}(1-\epsilon t)\right)dt \\
&=-\lim_{T\to \infty} e^{-\epsilon T}(1-\epsilon T)\int_{[\frac{1}{m},m]^d } { b_{\kappa_m,  p}}(x,T)x^{2\alpha}dx\\
&\quad +\lim_{t\to 0^+}\int_{[\frac{1}{m},m]^d } { b_{\kappa_m,  p}}(x,t)x^{2\alpha}dx   \\
& \quad +\int_0^\infty \int_{[\frac{1}{m},m]^d }  { b_{\kappa_m,  p}}(x,t)x^{2\alpha}dx (\epsilon^2t-2\epsilon)e^{-\epsilon t}dt  \\
& \le \lim_{T\to \infty} e^{-\epsilon T}(\epsilon T-1)\int_{[\frac{1}{m},m]^d } { b_{\kappa_m,  p}}(x,T)x^{2\alpha}dx \\
&\quad +\lim_{t\to 0^+}\int_{[\frac{1}{m},m]^d } { b_{\kappa_m,  p}}(x,t)x^{2\alpha}dx   \\
& \quad +\int_0^\infty \int_{[\frac{1}{m},m]^d }  { b_{\kappa_m,
p}}(x,t)x^{2\alpha}dx \epsilon^2te^{-\epsilon t}dt,
\end{align*}
where we have used again integration by parts and the fact that ${
b_{\kappa_m,  p}}$ is non-negative.

The function $P_t^\alpha f(x)$ is continuous in $(x,t)\in
(0,\infty)\times [0,\infty)$ and $\lim_{t\to 0^+}P_t^\alpha f(x)$
$=f(x)$, $x\in (0,\infty)$. Also, we have that the function
$\mathbb{P}_t^{\alpha,i} {g_i}(x)$ is continuous in $(x,t)\in
(0,\infty)\times [0,\infty)$ and $\lim_{t\to
0^+}\mathbb{P}_t^{\alpha,i} {g_i}(x)=\lim_{t\to
0^+}x_iP_t^{\alpha+e_i}({g_i}/._i)(x)={g_i}(x)$, $x\in (0,\infty)${,
for each $i=1,\dots,d$}. We obtain then
\begin{align*}
\lim_{t\to 0^+}\int_{[\frac{1}{m},m]^d } { b_{\kappa_m,  p}}(x,t)x^{2\alpha}dx&=\int_{[\frac{1}{m},m]^d } { b_{\kappa_m,  p}}(x,0)x^{2\alpha}dx\\
&=\int_{[\frac{1}{m},m]^d }B_{\kappa_m,p}(f(x),g(x))x^{2\alpha}dx.
\end{align*}

Therefore,
\begin{align*}
I_{m,\epsilon}&\leq \lim_{T\to \infty} T e^{-\epsilon T}\int_{[\frac{1}{m},m]^d }\partial_t { b_{\kappa_m,  p}}(x,T)x^{2\alpha}dx\\
&\quad - \lim_{t\to 0^+}te^{-\epsilon t}\int_{[\frac{1}{m},m]^d } \partial_t { b_{\kappa_m,  p}}(x,t)x^{2\alpha}dx   \\
& \quad +\lim_{T\to \infty} e^{-\epsilon T}(\epsilon T-1)\int_{[\frac{1}{m},m]^d } { b_{\kappa_m,  p}}(x,T)x^{2\alpha}dx\\
&\quad +\int_{[\frac{1}{m},m]^d } B_{\kappa_m, p}(f(x),g(x))x^{2\alpha}dx   \\
& \quad +\int_0^\infty \int_{[\frac{1}{m},m]^d }  { b_{\kappa_m,
p}}(x,t)x^{2\alpha}dx \epsilon^2te^{-\epsilon
t}dt=:\sum\limits_{j=1}^5 I_{m,\epsilon}^j.
\end{align*}
Let us prove that
$I_{m,\epsilon}^1=I_{m,\epsilon}^2=I_{m,\epsilon}^3=0$. Since
$\int_{\mathbb{R}^d_+}P_t^\alpha(x,y)y^{2\alpha}dy=1$, for every
$x\in \mathbb{R}^d_+$ and $t>0$,  $|t\partial_tP_t^\alpha(x,y)|\le
CP_{t/2}^\alpha(x,y)$, $x,y\in \mathbb{R}_+^d$ and $t>0$, we deduce
that $|P_t^\alpha(f)(x)|+|t\partial_tP_t^\alpha(f)(x)|\le
C\|f\|_\infty$, for every $x\in \mathbb{R}^d_+$ and $t>0$. Also, by
\eqref{Z1} and \eqref{P<PP}, it follows that
$|\mathbb{P}_t^{\alpha,i}({g_i})(x)|+|t\partial_t\mathbb{P}_t^{\alpha,i}({g_i})(x)|\le
C\|{g_i}\|_\infty$, for every {$i=1,\dots,d$,} $x\in \mathbb{R}^d_+$
and $t>0$. Then, from \ref{partialBk},  we have
\begin{align*}
t|\partial_t  { b_{\kappa_m,  p}}(x,t)|&\le {C} \left(|P_t^\alpha f(x)|^{p-1}+{|\mathbb P_t^{\alpha}g(x)|}+{ |\mathbb P_t^{\alpha}g(x)|^{p'-1}}+\kappa_m^{p'-1}\right)\\
&\quad \times\left(|t\partial_t P_t^\alpha f(x)|+{ |t\partial_t\mathbb P_t^{\alpha}g(x)|}\right)  \\
&\le C \left(\|f\|_\infty^{p-1}+{\left(\sum\limits_{i=1}^d
\|g_i\|_\infty^2\right)^{1/2}}+{\left(\sum\limits_{i=1}^d
\|g_i\|_\infty^2\right)^{\frac{p'-1}{2}}}+\kappa_m^{p'-1}\right)
\\
&\quad \times\left(|t\partial_t P_t^\alpha f(x)|+{ |t\partial_t\mathbb P_t^{\alpha}g(x)|}\right)  \\
&\le {C}
\left(\|f\|_\infty^{p-1}+{\left(\sum\limits_{i=1}^d \|g_i\|_\infty^2\right)^{1/2}}+{\left(\sum\limits_{i=1}^d \|g_i\|_\infty^2\right)^{\frac{p'-1}{2}}}+\kappa_m^{p'-1}\right)\\
&\quad \times \left(\|f\|_\infty+{\left(\sum\limits_{i=1}^d
\|g_i\|_\infty^2\right)^{1/2}}\right),
\end{align*}
{for $x\in \mathbb R^d_+$ and $t>0$.}

Then, from \ref{tdP-t0}, \ref{tdP-tinfty}, \ref{tdPP-t0} and
\ref{tdPP-tinfty}, we deduce that
\[
t|\partial_t { b_{\kappa_m,  p}}(x,t)|\to 0,
\]
whenever $t\to \infty$ or $t\to 0^+$.

Thus, by applying the Dominated Convergence Theorem we get that
$I_{m,\epsilon}^1=I_{m,\epsilon}^2=0$.

Let us take care of $I_{m,\epsilon}^3$. In this case, by using \ref{cotaBk}, for $T>\epsilon^{-1}$, we have
\begin{align*}
&(\epsilon T-1)e^{-\epsilon T}\int_{[\frac{1}{m},m]^d} { b_{\kappa_m,  p}}(x,T)x^{2\alpha}dx\\
&\quad\leq C \epsilon T e^{-\epsilon T}\int_{[\frac{1}{m},m]^d} \left[ (|P^\alpha_T f(x)|+\kappa_m)^p +({|\mathbb{P}^{\alpha}_T g(x)|}+\kappa_m)^{p'} \right]x^{2\alpha}dx\\
& \quad\leq C \epsilon T e^{-\epsilon T}\left(\|f\|_p^p+{\left\VERT
g\right\VERT
_{p'}^{p'}}+m^{2|\alpha|+d}(\kappa_m^p+\kappa_m^{p'})\right)\longrightarrow
0,
\end{align*}
as $T\to \infty$. Hence, $I_{m,\epsilon}^3=0$ which yields that
\begin{align*}
I_{m,\epsilon}&\leq I_{m}^4+I_{m,\epsilon}^5\\
&=\int_{[\frac{1}{m},m]^d }B_{\kappa_m,p}(f(x),g(x))x^{2\alpha}dx
+\int_0^\infty \int_{[\frac{1}{m},m]^d}  b_{\kappa_m,
p}(x,s/\epsilon)x^{2\alpha}dx se^{-s}ds.
\end{align*}
Following the ideas in \cite[p.~15]{W2}, from \ref{cotaBk} and the
choice of $\kappa_m$, we have that
\begin{align*}
I_{m}^4&\le \frac{1+\gamma(p)}{2}\left((1+\epsilon)^p\|f\|_p^p+(1+\epsilon)^{p'}{\left\VERT g\right\VERT _{p'}^{p'}}\right.\\
&\quad\left.+(1+\epsilon^{-1})^p\kappa_m^p m^{2|\alpha|+d}+(1+\epsilon^{-1})^{p'}\kappa_m^{p'} m^{2|\alpha|+d}\right)\\
&\leq
\frac{1+\gamma(p)}{2}\left((1+\epsilon)^p\|f\|_p^p+(1+\epsilon)^{p'}{\left\VERT
g\right\VERT
_{p'}^{p'}}+\kappa_m^{1/2}\left[(1+\epsilon^{-1})^p+(1+\epsilon^{-1})^{p'}\right]\right).
\end{align*}
Then, since $\lim_{m\rightarrow \infty} \kappa_m=0$, we obtain
\[
\limsup_{m\to \infty}I_{m}^4\le
\frac{1+\gamma(p)}{2}\left(\|f\|_p^p+{\left\VERT g\right\VERT
_{p'}^{p'}}\right).
\]

In order to estimate $I_{m,\epsilon}^5$, we apply \ref{cotaBk}, \eqref{P<PP} and the properties on $\kappa_m$, to get
\begin{align*}I_{m,\epsilon}^5& \le (1+\gamma(p))\left[2^{p-1} \left(\int_0^\infty \int_{[\frac{1}{m},m]^d}|P_{s/\epsilon}^\alpha f(x)|^px^{2\alpha}dxse^{-s}ds+\kappa_m^{p}m^{2|\alpha|+d}\right)\right.\\
&\quad\left. +2^{p'-1} \left(\int_0^\infty \int_{[\frac{1}{m},m]^d}{|\mathbb{P}^{\alpha}_{s/\epsilon} g(x)|^{p'}} x^{2\alpha}dxse^{-s}ds+\kappa_m^{p'}m^{2|\alpha|+d}\right)\right]\\
& \le (1+\gamma(p)) 2^{p-1} \int_0^\infty \int_{[\frac{1}{m},m]^d}\left(|P_{s/\epsilon}^\alpha f(x)|^p+{|P^{\alpha}_{s/\epsilon} g(x)|^{p'}}\right)x^{2\alpha}dx se^{-s}ds \\
&\quad +(1+\gamma(p)) 2^{p-1}\kappa_m^{1/2},
\end{align*}
{with $|P^{\alpha}_{s/\epsilon} g(x)|=\left(\sum_{i=1}^d
|P^{\alpha}_{s/\epsilon} g_i(x)|^2\right)^{1/2}$ and $|\mathbb
P^{\alpha}_{s/\epsilon} g(x)|=\left(\sum_{i=1}^d |\mathbb
P^{\alpha,i}_{s/\epsilon} g_i(x)|^2\right)^{1/2}$.} Thus
\[\limsup_{m\to \infty}I_{m,\epsilon}^5\le (1+\gamma(p)) 2^{p-1}\int_0^\infty \int_{\mathbb R^d_+}\left(|P^\alpha_{s/\epsilon}f(x)|^p+{|P^{\alpha}_{s/\epsilon} g(x)|^{p'}}\right)x^{2\alpha}dx se^{-s}ds.\]
If we define
\[H_\epsilon (x,s)=
\left(|P^\alpha_{s/\epsilon}f(x)|^p+{|P^{\alpha}_{s/\epsilon}
g(x)|^{p'}}\right) se^{-s},\,\,\,x\in
\mathbb{R}_+^d\,\,\,\mathrm{and}\,\,\,s>0,\] we observe that
$\lim_{\epsilon\to 0^+}H_\epsilon (x,s)=0$ for every $ x\in \mathbb
R_+^{d}$ and $s>0$, by virtue of \ref{P-tinfty}. On the other hand,
since the maximal operator $P^\alpha_{*}f=\sup_{t>0} | P^\alpha_t f|
$ is bounded on $L^q(\mathbb R^d_+,x^{2\alpha}dx)$ for every
$1<q\le \infty$ by virtue of \cite[Proposition~6.2]{NS}
and the fact that $\|P_*^\alpha f\|_q\leq \|T_*^\alpha f\|_q\le
C\|f\|_q$, for $1< q\le \infty,$ we get
that
\[H_\epsilon (x,s)\le \left(|P^\alpha_{*}f(x)|^p+\sum\limits_{i=1}^d|P^{\alpha}_{*} g_i(x)|^{p'}\right) se^{-s}\in L^1(\mathbb R^{d+1}_+,x^{2\alpha}dxds).\]

Finally, by applying the Dominated Convergence Theorem we obtain that \[\limsup_{\epsilon\to 0^+}\limsup_{m\to \infty}I_{m,\epsilon}^5=0.\]
Therefore
\[\limsup_{\epsilon\to 0^+}\limsup_{m\to \infty}\int_0^\infty  \int_{[\frac{1}{m},m]^d} \partial_t^2 { b_{\kappa_m,  p}}(x,t)x^{2\alpha}dx e^{-\epsilon t}tdt\le \frac{1+\gamma(p)}{2}(\|f\|_p^p+{\left\VERT g\right\VERT_{p'}^{p'}}).\]

{\bf Estimation of $J_{m,\epsilon,\ell}$:} Let $\ell=1,\ldots,d$. We shall see that
\[\limsup_{\epsilon\to 0^+}\limsup_{m\to \infty}J_{m,\epsilon, \ell}{\leq 0.}\]
Actually, we will show that for every
$\epsilon>0$,
\begin{equation}
\limsup_{m\to \infty}J_{m,\epsilon, \ell}\leq 0.
\end{equation}

In order to do so, we integrate by parts on the $\ell$-th variable
to get {\begin{align*} J_{m,\epsilon,\ell} &
=\int_0^\infty\int_{[\frac{1}{m},m]^{d-1}}
m^{2\alpha_\ell}\partial_{x_\ell}  b_{\kappa_m,
p}(m,\widehat{x}^\ell,
t)(\widehat{x}^\ell)^{2\widehat{\alpha}^\ell}d\widehat{x}^\ell\
e^{-\epsilon t}tdt\\
&\quad-\int_0^\infty\int_{[\frac{1}{m},m]^{d-1}}(1/m)^{2\alpha_\ell}\partial_{x_\ell}
b_{\kappa_m,
p}(1/m,\widehat{x}^\ell,t)(\widehat{x}^\ell)^{2\widehat{\alpha}^\ell}d\widehat{x}^\ell\
e^{-\epsilon t}tdt\\
&=:J_{m,\epsilon,\ell}^1-J_{m,\epsilon,\ell}^2.
\end{align*}

Then, we will see that
\[\limsup_{m\to \infty}J_{m,\epsilon, \ell}^1-\liminf_{m\to \infty}J_{m,\epsilon, \ell}^2\leq 0.\]}

Let us {first} remark that from \ref{partialBk}, \eqref{P<PP}, and
the boundedness of $P_t^\alpha $ on $L^\infty(\mathbb R^d_+)$, we
get, { for every $x\in \mathbb R^d_+$ and $t>0$,}
\begin{align*}
\nonumber x_\ell^{2\alpha_\ell}|\partial_{x_\ell} {b_{\kappa_m,  p}}&(x,t)|\\
&\le C_p  \left[\left(|P_t^\alpha f(x)|+\kappa_m\right)^{p-1}+|\mathbb P_t^{\alpha}g(x)|+\kappa_m\right]x_\ell^{2\alpha_\ell}|\partial_{x_\ell} P_t^\alpha f(x)| \\
&\quad +C_p\left({|\mathbb P_t^{\alpha}g(x)|}+\kappa_m\right)^{p'-1}x_\ell^{2\alpha_\ell}{ |\partial_{x_\ell}\mathbb P_t^{\alpha}g(x)|}\\
&\le C_p  \left[\left(|P_t^\alpha f(x)|+\kappa_m\right)^{p-1}+{| P_t^{\alpha}g(x)|}+\kappa_m\right]x_\ell^{2\alpha_\ell}|\partial_{x_\ell} P_t^\alpha f(x)| \\
&\quad +C_p\left({| P_t^{\alpha}g(x)|}+\kappa_m\right)^{p'-1}x_\ell^{2\alpha_\ell}{ |\partial_{x_\ell}\mathbb P_t^{\alpha}g(x)|}\\
&\le C_p \left[(\|f\|_\infty+\kappa_m)^{p-1}+{\left(\sum\limits_{i=1}^d \|g_i\|_\infty^2\right)^{1/2}}+\kappa_m\right]x_\ell^{2\alpha_\ell}|\partial_{x_\ell} P_t^\alpha f(x)| \\
&\quad +C_p\left({\left(\sum\limits_{i=1}^d
\|g_i\|_\infty^2\right)^{1/2}}+\kappa_m\right)^{p'-1}x_\ell^{2\alpha_\ell}{
|\partial_{x_\ell}\mathbb P_t^{\alpha}g(x)|}.
\end{align*}

{In order to estimate $J_{m,\epsilon,\ell}^1$,} we consider the
functions
\[
A(x,t)=x_\ell^{2\alpha_\ell}|\partial_{x_\ell}P_t^\alpha(f)(x)|,\,\,\,x\in
\mathbb{R}_+^d\,\,\, \mathrm{and} \,\,\,t>0,
\]
and, {for $i=1,\dots, d$,
\begin{align*}
B_i(x,t)&=x_\ell^{2\alpha_\ell}|\partial_{x_\ell}\mathbb{P}_t^{\alpha,i}(g_i)(x)|,\,\,\,x\in
\mathbb{R}_+^d\,\,\, \mathrm{and} \,\,\,t>0.
\end{align*}}
According to \ref{Des2} we get
\begin{align*}
A(x,t)&\le Cx_\ell^{2\alpha_\ell}\int_{[a,b]^d}\frac{t(x_\ell+1)}{(t^2+|x-y|^2)^{|\alpha|+\frac{d+3}{2}}}y^{2\alpha}dy\\
&\le Ctx_\ell^{2\alpha_\ell+1}\int_{[a,b]^d}\frac{1}{(x_\ell^2+|\hat{x}^\ell-\hat{y}^\ell|^2)^{|\alpha|+\frac{d+3}{2}}}dy\\
&\le
C\frac{t}{x_\ell}\int_{[a,b]^d}\frac{1}{(1+|\hat{x}^\ell-\hat{y}^\ell|^2)^{|\hat{\alpha}^\ell|+\frac{d+1}{2}}}dy,
\end{align*}
when $t>0$, $x\in \mathbb{R}^d_+$ and $x_\ell>2b$, and
\begin{align*}
A(x,t)&\le Cx_\ell^{2\alpha_\ell}\int_{[a,b]^d}\frac{t(x_\ell+1)}{(t^2+|x-y|^2)^{|\alpha|+\frac{d+3}{2}}}y^{2\alpha}dy\\
&\le
Ctx_\ell^{2\alpha_\ell}\int_{[a,b]^d}\frac{1}{(1+|\hat{x}^\ell-\hat{y}^\ell|^2)^{|\hat{\alpha}^\ell|+\frac{d+3}{2}}}dy,
\end{align*}
when $t>0$, $x\in \mathbb{R}^d_+$ and $x_\ell<a/2$.

By \eqref{Z1} we get
\begin{equation}\label{derivPPi}
\partial_{x_\ell}\mathbb{P}_t^{\alpha,i}(g)(x)=\delta_{i,\ell}P_t^{\alpha+e_i}\left(\frac{g}{._i}\right)(x)+x_i\partial_{x_\ell}P_t^{\alpha+e_i}\left(\frac{g}{._i}\right)(x),
\end{equation}
for $t>0\,\,\,\mathrm{and}\,\,\,x\in \mathbb{R}_+^d.$
Here, $\delta_{i,\ell}=0$, when $i\neq \ell$, and
$\delta_{i,\ell}=1$, when $i=\ell$. Then, by using \ref{Des1} and
\ref{Des2} and proceeding as above, we obtain
\begin{align*}
B_i(x,t)&\le Cx_\ell^{2\alpha_\ell}\Bigg(\int_{[a,b]^d}\frac{t}{(t^2+|x-y|^2)^{|\alpha|+\frac{d+3}{2}}}y^{2\alpha}dy\\
&\quad + x_i\int_{[a,b]^d}\frac{t(x_\ell+1)}{(t^2+|x-y|^2)^{|\alpha|+\frac{d+5}{2}}}y^{2\alpha}dy\Bigg)\\
&\le
Ct\Bigg(\frac{1}{x_\ell}+\frac{x_i}{x_\ell^2}\Bigg)\int_{[a,b]^d}\frac{1}{(1+|\hat{x}^\ell-\hat{y}^\ell|^2)^{|\hat{\alpha}^\ell|+\frac{d+1}{2}}}dy
\end{align*}
when $t>0$, $x\in \mathbb{R}^d_+$ and $x_\ell>2b$, and
\begin{align*}
B_i(x,t) &\le
Ctx_\ell^{2\alpha_\ell}(1+x_i)\int_{[a,b]^d}\frac{1}{(1+|\hat{x}^\ell-\hat{y}^\ell|^2)^{|\hat{\alpha}^\ell|+\frac{d+1}{2}}}dy,
\end{align*}
when $t>0$, $x\in \mathbb{R}^d_+$ and $x_\ell<a/2$.

By the case $x_\ell >2 b$ in the estimates of
$A(x,t)$ and $B_i(x,t)$,  $i=1,\dots, d$, we get, for every
$\widehat{x}^\ell\in [1/m,m]^{d-1}$ and $t>0$, that
\begin{equation*}
m^{2\alpha_\ell} |\partial_{x_\ell}b_{\kappa_m,
p}(m,\widehat{x}^\ell,t)|\le
C(f,g_i,b)\frac{t}{m}\int_{[a,b]^d}\frac{1}{(1+|\hat{x}^\ell-\hat{y}^\ell|^2)^{|\hat{\alpha}^\ell|+\frac{d+1}{2}}}dy.
\end{equation*}
Then, $$ m^{2\alpha_\ell} |\partial_{x_\ell}b_{\kappa_m,
p}(m,\widehat{x}^\ell,t)|\to 0,
$$
 as $m\to \infty$. Also, the right-hand side of the above
inequality belongs to $L^1(\mathbb{R}^{d-1}_+\times (0,\infty),
(\widehat{x}^{\ell})^{2\hat{\alpha}^\ell}d\widehat{x}^{\ell}te^{-\epsilon
t}dt).$ Therefore $\limsup_{m\to\infty} J_{m,\epsilon, \ell}^1= 0.$

{We shall see now that $\liminf_{m\to\infty} J_{m,\epsilon,
\ell}^2\ge 0$. From \eqref{derivPPi} and the fact that
$P_t^{\alpha+e_i}$ is positive, we know that
\[
\partial_{x_\ell}\mathbb{P}_t^{\alpha,i}(g_i)(x)\geq
x_i\partial_{x_\ell}P_t^{\alpha+e_i}\left(\frac{g_i}{._i}\right)(x),
\]
for every $i=1,\dots, d$, $x\in \mathbb R^d_+$ and $t>0$.

On the other hand, \ref{partialBk} implies that $\partial_{\eta_i}
B_{\kappa_m,p}(u(x,t))\ge 0$, for every $i=1,\ldots,
d.$ This, together with the above inequality, yields

\begin{align*}
x_\ell^{2\alpha_\ell}&\partial_{x_\ell}  b_{\kappa_m,  p}(x,t)=x_\ell^{2\alpha_\ell}\partial_\zeta B_{\kappa_m,p}(u(x,t))\partial_{x_\ell} P_t^\alpha f(x)\\
&\quad +x_\ell^{2\alpha_\ell} \sum\limits_{i=1}^d \partial_{\eta_i} B_{\kappa_m,p}(u(x,t))\partial_{x_\ell}\mathbb{P}_t^{\alpha,i}(g_i)(x)\\
&\geq \partial_\zeta B_{\kappa_m,p}(u(x,t))x_\ell^{2\alpha_\ell}\partial_{x_\ell} P_t^\alpha f(x)\\
&\quad +\sum\limits_{i=1}^d \partial_{\eta_i} B_{\kappa_m,p}(u(x,t))
x_ix_\ell^{2\alpha_\ell}
\partial_{x_\ell}P_t^{\alpha+e_i}\left(\frac{g_i}{._i}\right)(x),\,\,\,x\in \mathbb{R}^d_+\,\,\,\rm{and}\,\,\,t>0.
\end{align*}
} Thus, {
\begin{align*}
J^2_{m,\epsilon, \ell}& \ge
\int_0^\infty\int_{[\frac{1}{m},m]^{d-1}}\left(1/m\right)^{2\alpha_\ell}\partial_{\zeta}B_{\kappa_m,p}(u(1/m,
\widehat{x}^\ell,t)) \\ & \quad \quad \quad \quad \quad \times
\partial_{x_\ell}P_t^\alpha f(1/m,
\widehat{x}^\ell)(\widehat{x}^\ell)^{2\widehat{\alpha}^\ell}d\widehat{x}^\ell\
e^{-\epsilon t}tdt\\ &\quad +\sum_{i=1}^d
\int_0^\infty\int_{[\frac{1}{m},m]^{d-1}}
\left(1/m\right)^{2\alpha_\ell}\partial_{\eta_i}B_{\kappa_m,p}(u(1/m,
\widehat{x}^\ell,t))\\ & \quad \quad \quad \quad \quad \times
x_i\partial_{x_\ell}P_t^{\alpha+e_i}
\left(\frac{g_i}{._i}\right)(1/m,
\widehat{x}^\ell)(\widehat{x}^\ell)^{2\widehat{\alpha}^\ell}d\widehat{x}^\ell\
e^{-\epsilon t}tdt\\ & \quad =:J_{m,\epsilon, \ell}^3
\end{align*}}
From property \ref{dP-x0} we get that the integrands of
$J_{m,\epsilon, \ell}^3$ converge pointwisely to $0$  as $m\to
\infty.$ On the other hand, in order to obtain an integrable
function independent of $m$ we proceed to bound the integrands as we
have done before, but now considering the estimates of $A$ and
$B_i$, $i=1,\dots, d,$ on the range $x_\ell<a/2.$ Then, by the
Dominated Convergence Theorem, we get that $\lim_{m\to
\infty}J_{m,\epsilon, \ell}^3=0.$ And so $\liminf_{m\to \infty}
J_{m,\epsilon, \ell}^2\ge 0.$

Bringing things together we have proved that
\[\left|\sum\limits_{i=1}^d\int_{\mathbb R_+^d}  R_{\alpha,i} f(x)g_i(x)  x^{2\alpha} dx\right|
\le 2\frac{1+\gamma(p)}{\gamma(p)}\left(\|f\|_p^p+\left\VERT g\right\VERT_{p'}^{p'}\right).\]
Let us take $\lambda>0$. By applying the above inequality to
$\lambda f$ and {$\frac{g_i}{\lambda}$, $i=1,\dots, d$,} we have
\[\bigg|\sum\limits_{i=1}^d\int_{\mathbb R_+^d}
R_{\alpha,i} f(x) g_i(x)  x^{2\alpha} dx\bigg|\leq
2\frac{1+\gamma(p)}{\gamma(p)}\left(\lambda^p
\|f\|_p^p+\frac{1}{\lambda^{p'}}\left\VERT
g\right\VERT_{p'}^{p'}\right),\] and minimizing on $\lambda$ yields
{\begin{align*}
\bigg|\sum\limits_{i=1}^d\int_{\mathbb R_+^d}  R_{\alpha,i} f(x) g_i(x)  x^{2\alpha} dx\bigg|&\leq 2\frac{1+\gamma(p)}{\gamma(p)} \left(\left(\frac{p}{p'}\right)^{1/p}+\left(\frac{p'}{p}\right)^{1/p'}\right)\|f\|_p\left\VERT g\right\VERT_{p'}\\
&:=4D_p\|f\|_p\left\VERT g\right\VERT_{p'}.
\end{align*}}
Approximation arguments allow us to see that the last inequality
also holds when $0\le {g_i}\in L^p(\mathbb{R}^d_+,x^{2\alpha}dx)$,
{$i=1,\dots, d$} and $f\in C_c^\infty(\mathbb{R}_+^d)$. Therefore,
if $f\in C_c^\infty(\mathbb{R}_+^d)$ and ${g_i}\in
L^p(\mathbb{R}^d_+,x^{2\alpha}dx)$ for every {$i=1,\dots, d$}, by
considering {$g_i=g_i^+-g_i^-$, where $g_i^+=\max\{g_i,0\}$,} we
deduce
\[{\|R_{\alpha} f\|}_{L^p(\mathbb R^d_+,x^{2\alpha}dx)}\leq 8D_p\|f\|_{L^p(\mathbb R^d_+,x^{2\alpha}dx)}.\]
Finally, as it was proved in \cite[p. 761]{W2}, the constant $D_p$
is smaller than $6(p^*-1)$, which let us conclude with the proof in
this case.

In order to prove the result when $1<p<2$ we can proceed in a
similar way by interchanging the roles between $p$ and $p'$ and
between $P_t^\alpha(f)$ and $\mathbb{P}_t^{\alpha}(g)$. More
precisely, we need to consider $ b_{\kappa,p} (x,t)=
\tilde{B}_{\kappa,p}(\mathbb{P}_t^{\alpha}(g)(x),
P_t^\alpha(f)(x))$, where $\tilde{B}_{\kappa,p}$ is the Bellman
function defined on $\mathbb R^{d}\times \mathbb R$, that is,
$m_1=d$ and $m_2=1$. This function satisfies Proposition
\ref{bellman} with $p$ replaced by $p'$ and viceversa.

\vspace{3mm}

Theorem 1.2 {follows the same lines of \cite[Corollary 2]{DV2},} by
using above arguments and proceeding as in the proof of
\cite[Proposition 3]{DV2}.

\end{document}